\documentclass[11pt]{amsart}

\usepackage[table]{xcolor}

\usepackage{ytableau,tikz,varwidth}
\usetikzlibrary{calc, snakes}
\usepackage[enableskew]{youngtab}
\usepackage{geometry}
\usepackage{enumerate,}

\newtheorem{thm}{Theorem}[section]
\newtheorem{lemma}[thm]{Lemma}
\newtheorem{prop}[thm]{Proposition}
\newtheorem{claim}[thm]{Claim}
\newtheorem{cor}[thm]{Corollary}

\theoremstyle{definition} 
\newtheorem{ex}[thm]{Example}
\newtheorem{defn}[thm]{Definition}

\theoremstyle{remark}
\newtheorem{remark}[thm]{Remark}

\newcommand{\lam}{\lambda}
\newcommand{\lamu}{\sigma}

\newcommand{\al}{\alpha}
\newcommand{\be}{\beta}
\newcommand{\sig}{\sigma}
\newcommand{\E}{\mathbb{E}}

\newcommand{\PP}{\mathbb{P}}
\newcommand{\cH}{\mathcal{H}}
\newcommand{\cL}{\mathcal{L}}
\newcommand{\cO}{\mathcal{O}}
\newcommand{\cP}{\mathcal{P}}
\newcommand{\cQ}{\mathcal{Q}}
\newcommand{\cU}{\mathcal{U}}
\newcommand{\cV}{\mathcal{V}}
\newcommand{\BN}{\textrm{BN}}
\newcommand{\comment}[1]{}

\newcommand{\Grd}{G^r_d}
\newcommand{\grd}{g^r_d}
\newcommand{\Grdab}{G^{r,\alpha,\beta}_d}

\newcommand{\bs}{\boldsymbol}
\newcommand{\ul}{\underline}
\newcommand{\st}{\scriptstyle}

\DeclareMathOperator{\Pic}{Pic}
\DeclareMathOperator{\VS}{VS}

\title[Genera of Brill-Noether curves and Young tableaux]{Genera of Brill-Noether curves and staircase paths in Young tableaux}
\author[M. Chan]{Melody Chan}\address{Department of Mathematics, Brown University, Box
1917, Providence, RI 02912}\email{mtchan@math.brown.edu}
\author[A. L\'opez]{Alberto L\'opez Mart\'in}\address{IMPA, Estrada Dona Castorina, 110, Rio de Janeiro, RJ 22460-902}
\email{alopez@impa.br}
\author[N. Pflueger]{Nathan Pflueger}\address{Department of Mathematics, Brown University, Box
1917, Providence, RI 02912}\email{pflueger@math.brown.edu}
\author[M. Teixidor]{Montserrat Teixidor i Bigas}\address{Department of Mathematics, Tufts University, Medford, MA 02155}
\email{montserrat.teixidoribigas@tufts.edu}
\date{\today}

\begin{document}
\maketitle

\begin{abstract}
In this paper,  we compute the genus of the variety of linear series of rank $r$ and degree $d$ on a general curve of genus $g$, with ramification at least $\alpha$ and $\beta$ at two given points, when that variety is 1-dimensional.  Our proof uses degenerations and limit linear series along with an analysis of random staircase paths in Young tableaux, and produces an explicit scheme-theoretic description of the limit linear series of fixed rank and degree on a generic chain of elliptic curves when that scheme is itself a curve.  
\end{abstract}
\section{Introduction}

Fix numbers $g,r,d$, and let $X$ be a smooth, proper curve of genus $g$ over an algebraically closed field.
A {\em linear series} of rank $r$ and degree $d$ on $X$, or a {\em $g^r_d$}  for short, is the pair of a line bundle $L\in \Pic^d(X)$
 together with an $(r\!+\!1)$-dimensional space $V\subseteq H^0(X,L)$. 
  Linear series are the central object of study in the classical Brill-Noether theory of algebraic curves. 
   For example, the main results of Brill-Noether theory imply that when 
$$\rho(g,r,d) := g-(r+1)(g-d+r)$$
is nonnegative, the $g^r_d$s on a general curve $X$ of genus $g$ form a proper scheme $G^r_d(X)$ that is smooth of expected dimension $\rho$, 
and connected if $\rho>0$ \cite{fl-connected,gieseker,gh-brillnoether}. 
Thus, if $\rho=1$, then $\Grd(X)$ is a smooth, proper curve, whose genus $g'$ is known:

\begin{thm}\label{t:intro}
Suppose $\rho(g,r,d) = 1$.  For a general smooth curve $X$ of genus $g$, the genus of the curve $\Grd(X)$ is 
\begin{equation}
\label{eq:formula}
g'=1+\frac{(r+1)(g-d+r)}{g-d+2r+1}\cdot g!\cdot \prod_{i=0}^r \frac{i!}{(g-d+r+i)!}.
\end{equation}
\end{thm}

\noindent This result is due to Eisenbud-Harris \cite{eh-generaltype} and Pirola \cite{pirola}; the case $r=1$ had been proven previously by Kempf \cite{kempf}. 
 Theorem~\ref{t:intro} is not a mere curiosity.  It features as an ingredient in the proof of the main theorem of \cite{eh-generaltype}.  
 Moreover, the rational  map $\overline{\mathcal{M}_g} \dashrightarrow \overline{\mathcal{M}_{g'}}$ of moduli spaces obtained by assigning to a curve 
 its Brill Noether curve for suitable values of $r, d$ is  exploited in recent work of Farkas \cite{farkas} and Ortega \cite{ortega}.

In this paper, we give a new proof of Theorem~\ref{t:intro}, and we generalize it to the case of curves parametrizing $g^r_d$s on $X$ 
with prescribed ramification profiles at two generic fixed points of $X$.

\begin{thm}\label{t:intro2}
Fix $g,r,$ and $d$, and let $\alpha = (\alpha_0,\ldots, \alpha_r)$ be a nondecreasing and  $\beta = (\beta_0,\ldots,\beta_r)$ a nonincreasing sequence of integers.
 Let $\sig = \sig(g,r,d,\al,\be)$ be the skew shape defined by $(g,r,d,\al,\be)$ as in Definition \ref{def:skewtab}.
Suppose that the adjusted Brill Noether number is
\begin{equation}\label{eq:adjustedrho}
\rho(g,r,d,\al,\be) = g-(r+1)(g-d+r)-|\al|-|\be| = 1.
\end{equation}
Then for a general twice-pointed smooth curve $(X,p,q)$ of genus $g$, the scheme $\Grdab(X,p,q)$ is a curve, with at most nodes as singularities, of arithmetic genus 
\begin{equation}\label{eq:main}
1+(r\!+\!1)(n\!+\!1)f^\sigma + \sum_{i=1}^{r+1} (r\!+\!1\!-\!i)\!\cdot\! f^{{}^i\sigma} - \sum_{i=1}^{r+1} (r\!+2\!-\!i)\!\cdot\! f^{\sigma^i},
\end{equation}
where $f^\sigma$ denotes the number of standard fillings of the skew shape $\sigma$ in Definition \ref{def:skewtab}, 
and where ${\sigma^i}$ and ${{}^i\sigma}$ refer to shapes closely related to $\sigma$ (see Definition~\ref{d:sigmai}).
\end{thm}
\begin{figure}
\begin{tikzpicture}[inner sep=0in,outer sep=0in]
\node (s) {\begin{varwidth}{5cm}{
\tiny
\begin{ytableau}
\none & \none & \none & *(white)& &&&&*(black!20!white)&*(black!20!white) \\
\none & \none & \none & *(white)& &&&&*(black!20!white) \\
\none &*(blue!30!white) &*(blue!30!white) & *(white)& &&& \\
 *(blue!30!white)& *(blue!30!white) &*(blue!30!white)  & & &&&
\end{ytableau}
}\end{varwidth}};
\draw[snake=brace,thick] 	(-0.7,1.2) -- (1.55,1.2);
\draw[snake=brace,thick] (-2.5,-1.1) -- (-2.5,1);
\draw (-.3,1.37) node(alpha1)[right] {\Tiny$g-d+r$};
\draw (-3.3,-0.05) node(alpha1)[right] {\Tiny$r+1$};
\end{tikzpicture}
\caption{Skew shape $\sigma(g,r,d,\alpha,\beta)$ associated to the data $(g,r,d)=(6,3,4)$, $\alpha$=(0,0,2,3), $\beta$=(2,1,0,0). }
\label{fig:grdab}
\end{figure}
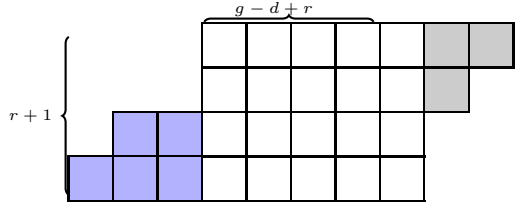
The construction  of the shape $\sigma$ is illustrated by an example in Figure~\ref{fig:grdab}.  The numbers $f^{\sigma^i}$ and $f^{{}^i\sigma}$  defined in Definition~\ref{d:sigmai}  are explicitly calculable using the determinantal formula~\eqref{eq:aitken}.  When $\alpha=\beta=0$, our formula reduces directly to~\eqref{eq:formula} (see Corollary~\ref{cor:final}).

The original proofs of Theorem~\ref{t:intro} were based on computations of cohomology classes in the Jacobian. 
Our proofs use degeneration  and limit linear series  building on the techniques introduced by
Castorena-L\'opez-Teixidor in \cite{castorena-lopez-teixidor}, who considered the case $r=1$ and no ramification points. 
We enumerate components of the space of limit linear series on a chain of elliptic curves  according to ramification data at the nodes.
This type of curve was introduced by Welters  in \cite{welters}  and used successfully thereafter in a number of applications 
(see \cite{castorena-teixidor}, \cite{Nathan}, \cite{Tei},  \cite{Tei2}, \cite{Exvb}, \cite{Tei3}, \cite{Tei4}).
In this paper, we  prove Theorem~\ref{t:intro} for all $r$ and prescribed ramification at two points. 
We also give an original proof of the reducedness of the special fiber of the relevant degeneration to limit linear series when $\rho =1$.  
The proof is not based on the standard study of the Gieseker-Petri map,  but relies instead on the representability of the Brill-Noether functor.
Results of Osserman and Murray-Osserman on the comparison of Osserman and Eisenbud-Harris linear series developed in \cite{murrayosserman, osserman-functor},
 allow us to deduce the genus of the Brill-Noether locus of the general curve from the genus of the Brill-Noether locus of the degeneration.

Brill-Noether loci with fixed ramification have not been studied much so far. We expect to come back to this topic in the future. 
A direct generalization of our results to more than two points of ramification will require considerable more work and might fail in some cases in positive characteristic, 
 as  the Brill-Noether dimension estimate may no longer apply in this case \cite[Remark~2.7.8]{osserman}. 

Our methods also apply in the case $\rho(g,r,d,\alpha,\beta) = 0$,
where they give the following enumerative geometry result. 
This result was already deduced by Tarasca \cite[Section 3.1]{t} using the formula  \cite[14.7.11(v)]{fulton-int-theory} for intersecting Schubert classes.
In recent work, Farkas and Tarasca \cite{farkas-tarasca} also consider a similar problem where a single ramification point is allowed to move.
Our proof generalizes the proof given in \cite{castorena-lopez-teixidor} in the case of trivial ramification and makes the role of skew tableaux explicit.

\begin{thm}\label{t:intro3}
Fix $g,r,d,\alpha,\beta$, and $\rho(g,r,d,\alpha,\beta)$ be defined as in Theorem \ref{t:intro2}, and assume that
$$\rho(g,r,d,\alpha,\beta) = 0.$$
Then for a general twice-pointed smooth curve $(X,p,q)$ of genus $g$, the scheme $G^{r,\alpha,\beta}_d(X,p,q)$ consists of $f^\sigma$ reduced points.  
\end{thm}

As we mentioned above, our strategy in the proof of  Theorem~\ref{t:intro2} is to compute the arithmetic genus of the scheme of Eisenbud-Harris limit linear series 
on a generic elliptic chain. We describe the main ingredients of the proof below.

Our aim in  Sections ~\ref{sec:combinatorics} and  \ref{sec:ponttab}  is to develop  the combinatorial techniques we need in the proof of our main Theorems,
 but some   of the results (see  Theorem~\ref{t:masterformula}, Lemma~\ref{p:upperright}, and Corollary~\ref{c:strangeidentity} ) can be of independent interest.  
In particular, we show  the following:  Pick a lattice path from the lower-left to upper-right corners of an $a\times b$ rectangle $\sigma$,
 with probability proportional to the number of standard compatible fillings of $\sigma$. 
 Then the expected number of turns in this path is exactly the harmonic mean of $a$ and $b$. 
 The set up of Section ~\ref{sec:combinatorics} is better expressed in terms of a graph (the Brill-Noether graph)
 associated to a generalization of Young tableaux called \textit{skew Young tableaux}.
  In Section \ref{sec:ponttab},  we introduce a further generalization of skew Young tableaux that we call  \textit{pontableaux} 
  and the corresponding subdivision of the Brill-Noether graph that we call  the \textit{augmented Brill-Noether graph}.
 \textit{Valid sequences}  relate pontableaux to the orders of vanishing of linear series.
  
  In Section~\ref{sec:elliptic}, we describe the schemes $\Grdab(E,p,q)$ parameterizing the space of linear series on an elliptic curve
with prescribed ramification profiles $\alpha,\beta$ at two points, in the case that that scheme has dimension at most 1.  
This gives a transversality result for two Schubert conditions on $\Grd(E),$ viewed as a Grassmann bundle over $\Pic^d(E)$.  
 Section~\ref{sec:prelim}  gives an explicit description of the scheme
 $\Grdab(X, p,q)$ of limit linear series on an elliptic chain $X$.  
 We show that $\Grdab(X,p,q)$ is a nodal curve with elliptic and rational components, whose dual graph is the augmented Brill-Noether graph. 
 
Finally, we use a result of \cite{murrayosserman} to conclude that the arithmetic genus of the nodal curve $\Grdab(X,p,q)$ of limit linear series 
coincides with the genus of the corresponding locus of linear series over a nearby smooth curve.  
Specifically, Osserman constructs a moduli functor for limit linear series \cite{osserman-functor}, in such a way that they form flat and proper families
 over one-parameter degenerations; 
 however, it is not always clear that the {\em scheme} structures on Osserman's limit linear series coincide with the Eisenbud-Harris scheme structures.
The paper \cite{murrayosserman} shows that under relatively mild conditions that are satisfied in our situation,
 the two scheme structures do coincide, and the statement on equality of genera follows.  This is described fully in Section~\ref{sec:final}.

We want to point out two directions for future work. First, our tools apply to Brill-Noether loci of any dimension
and could be used to study Brill-Noether loci of dimension bigger than one.
  Moreover, the description of the  locus of limit linear series on elliptic chains is very explicit and allows to compute invariants of the 
  Brill-Noether locus on the generic curve that are finer than the genus. Indeed, the aim of ~\cite{castorena-lopez-teixidor} was  to find the  gonality of the 
  $G^1_4(X)$ where $X$ is  a generic curve of genus 5.

\bigskip
\noindent {\bf Notational conventions.} We mention a few conventions here that will be used throughout the paper.
\begin{itemize}
\item The word ``variety'' refers to a finite-type reduced separated scheme over an algebraically closed field (not necessarily irreducible).
\item The numbers $g,r,d$ will always be nonnegative and satisfy $g-d+r \geq 0$.
\item The symbols $a_i,b_i$ will refer to vanishing orders of a linear series. The numbers $a_i$ will be in \textit{increasing} order, while the numbers $b_i$ will be in \textit{decreasing} order. The symbols $\alpha_i,\beta_i$ will be the corresponding ramification orders, defined by $\alpha_i = a_i -i,\ \beta_i = b_i + i - r$, and the symbols $\alpha,\beta$ will refer to the sequence of all of these numbers. In particular, $\alpha$ is nondecreasing and $\beta$ is nonincreasing.
\end{itemize}

\bigskip
\noindent {\bf Acknowledgments.}  We are very grateful to B.~Osserman for many useful remarks on this project, and for providing the reference to \cite{murrayosserman}. 
 We thank D.~Romik and I.~Pak for several useful pointers. MC thanks J.~Harris for helpful conversations, 
 and also thanks S.~Haddadan, S.~Hopkins, and L.~Moci for enlightening remarks that led to the final form of Theorem~\ref{t:masterformula}
  as well as the organizers of the AIM Workshop on Dynamical Algebraic Combinatorics for making those conversations possible.  
 We thank W.~Stein and SageMathCloud for providing a convenient platform for collaboration.  Finally, we thank the referee for thorough and valuable comments.  AL was supported by CAPES-Brazil. MC was supported by NSF DMS Award 1204278.  

\section{Staircase paths and Young tableaux}\label{sec:combinatorics}

The goal of this section is to count the number of vertices and edges in a particular graph that is related to Young tableaux, which we call the Brill-Noether graph. 
The main result is Theorem \ref{t:masterformula}. We will begin by reviewing the definitions of standard and skew-standard Young tableaux.  
Then we will come to our main new combinatorial definitions, of almost-standard tableaux and the Brill-Noether graph.
  To count the edges in the Brill-Noether graph, we will define and study a probability distribution on staircase paths in Young diagrams that seems to be new and interesting.

Throughout, let $n$ be a nonnegative integer.  A {\em partition} of $n$ is a tuple $\lambda = (\lambda_1, \ldots, \lambda_k)$ of positive integers, nonincreasing, summing to $n$. By convention, if $i > k$ then $\lambda_i = 0$.
The {\em Young diagram} of $\lambda$ is an array of boxes that has 
$\lambda_i$ boxes, left-justified, in the $i^{th}$ row (This convention for drawing diagrams is known as {\em English notation}).  We will often identify partitions with their Young diagrams and vice versa, without further mention.
The {\em conjugate} partition to $\lambda$ is the partition, denoted $\lambda^*$, obtained from $\lambda$ by reflecting over the diagonal.  In other words, $\lambda_i^*$ is the number of boxes in the $i^{th}$  column of $\lambda.$

A {\em standard Young tableau} of shape $\lambda$ is a filling of the Young diagram of $\lambda$
with the numbers $\{1,\ldots,n\}$, each appearing exactly once, such that the entries in each row and in each 
column are strictly increasing. See Figure~\ref{fig:staircasepath} for an example.  Write $SYT(\lambda)$ for the set of standard Young tableaux of shape $\lambda$.
Standard Young tableaux of a given shape $\lambda$ are counted by the celebrated hook-length formula, as follows.  Write $(i,j) \in \lambda$ for the box in the $i^{th}$ row and $j^{th}$ column of the Young diagram.  The {\em hook length} of box $(i,j)\in \lambda$ is defined as  
$$h(i,j) := \lam_i + \lam_j^* -i-j+1,$$
i.e.~$h(i,j)$ is one more than the number of boxes below $(i,j)$ plus the number of boxes to the right of $(i,j)$.  
Then by \cite{frt}, the number of standard Young tableaux of shape $\lambda$ is 
\begin{equation}\label{eq:hook}
f^\lam = \frac{n!}{\prod_{(i,j)\in \lam}h(i,j)}.
\end{equation}

\begin{defn}\label{d:skewYT} Consider two partitions $\lam$ and $\mu$ such that $\mu_i\,{\leq}\, \lambda_i$ for all $i$. 
The shape resulting from removing $\mu$ from $\lambda$ is called a \emph{skew Young diagram} or \emph{skew shape} and will be denoted $\lamu = \lambda\setminus\mu$.  
 We write $|\lamu|$ for the number of boxes in $\lamu$.
We say that $\lamu$ is {\em connected} if the lower left and upper right corners of $\lamu$ are connected by a walk along the edges of the boxes in the diagram. 

Suppose $|\lamu| = n$.  A \emph{skew standard Young tableau} of shape $\lamu$ is a bijective filling of the boxes of $\lamu$ with the numbers $\{1,\ldots,n\}$, such that the entries in each row and in each column are strictly increasing.
We will write $sSYT(\lamu)$ for the set of skew standard Young tableaux of shape $\lamu$, and write $f^{\lamu} = |sSYT(\lamu)|$.
\end{defn}
The number of  skew standard Young tableaux of shape $\lamu$, are counted by Aitken's determinantal formula (\cite{aitken}, see also \cite[Corollary 7.16.3]{ec2}):
\begin{equation}\label{eq:aitken}
f^{\lambda\setminus \mu} = |\lambda\!\setminus\!\mu|! \det \left(\frac{1}{(\lambda_i-i-\mu_j+j)!}\right)_{i,j \in \{1,2,\cdots,k\}}.
\end{equation}
Here $k$ denotes the number of parts of $\lambda$, and we interpret $1/m! = 0$ when $m$ is negative.  When $\mu=(\emptyset)$ then the formula~\eqref{eq:aitken} specializes to the more explicit hook-length product formula 
~\eqref{eq:hook}.
When we refer to skew Young diagrams, we include Young diagrams as the special case $\mu=(\emptyset)$.

Next we will define almost-standard Young tableaux and the Brill-Noether graph.

\begin{defn}\label{d:almost}
Let $\lamu$ be a skew shape with $n$ boxes.  We define an {\em almost-standard skew Young tableau} of shape $\lamu$ to be an injective numbering of the boxes of $\lamu$ with numbers chosen from $\{1,\ldots,n+1\}$, such that the entries in each row and in each column are strictly increasing.    Write $aSYT(\lamu)$ for the set of almost-standard skew Young tableaux of shape $\lamu$.   \end{defn}

\begin{defn}\label{d:bngraph}
Let $\lamu$ be a skew shape with $n$ boxes.  We define the {\em Brill-Noether graph} $BN(\lamu)$ as follows: the vertices of $BN(\lamu)$ are the almost-standard skew Young tableaux of shape $\lamu$, and two vertices are adjacent in $BN(\lamu)$ if they differ in exactly one box.
\end{defn}

\noindent Figure~\ref{fig:bngraph} shows the Brill-Noether graph on the $2\times 2$ square.

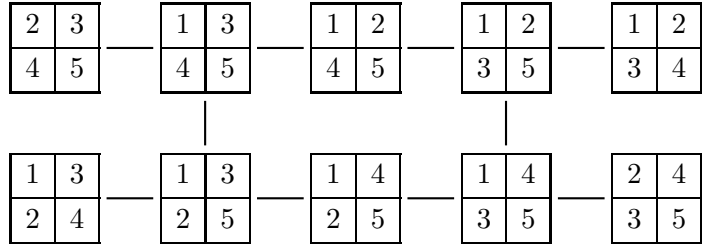
\begin{figure}
\begin{tikzpicture}[inner sep=0in,outer sep=0in]
\node (a) at (-6,3){\begin{varwidth}{5cm}{
\begin{ytableau}
2 & 3\\
4 & 5
\end{ytableau}}\end{varwidth}};

\node (b) at (-2,3){\begin{varwidth}{5cm}{
\begin{ytableau}
1 & 2\\
4 & 5
\end{ytableau}}\end{varwidth}};

\node (j) at (-4,3){\begin{varwidth}{5cm}{
\begin{ytableau}
1 & 3\\
4 & 5
\end{ytableau}}\end{varwidth}};

\node (c) at (0,3){\begin{varwidth}{5cm}{
\begin{ytableau}
1 & 2\\
3 & 5
\end{ytableau}}\end{varwidth}};

\node (d) at (2,3){\begin{varwidth}{5cm}{
\begin{ytableau}
1 & 2\\
3 & 4
\end{ytableau}}\end{varwidth}};

\node (e) at (2,1){\begin{varwidth}{5cm}{
\begin{ytableau}
2 & 4\\
3 & 5
\end{ytableau}}\end{varwidth}};

\node (f) at (0,1){\begin{varwidth}{5cm}{
\begin{ytableau}
1 & 4\\
3 & 5
\end{ytableau}}\end{varwidth}};

\node (g) at (-2,1){\begin{varwidth}{5cm}{
\begin{ytableau}
1 & 4\\
2 & 5
\end{ytableau}}\end{varwidth}};

\node (h) at (-4,1){\begin{varwidth}{5cm}{
\begin{ytableau}
1 & 3\\
2 & 5
\end{ytableau}}\end{varwidth}};

\node (i) at (-6,1){\begin{varwidth}{5cm}{
\begin{ytableau}
1 & 3\\
2 & 4
\end{ytableau}}\end{varwidth}};

\path[thick, black] ([xshift=3pt]a.east) edge  ([xshift=-3pt]j.west);
\path[thick, black] ([xshift=3pt]j.east) edge ([xshift=-3pt]b.west);
\path[thick, black] ([xshift=3pt]b.east) edge ([xshift=-3pt]c.west);
\path[thick, black] ([xshift=3pt]c.east) edge ([xshift=-3pt]d.west);

\path[thick, black] ([xshift=3pt]i.east) edge ([xshift=-3pt]h.west);
\path[thick, black] ([xshift=3pt]h.east) edge ([xshift=-3pt]g.west);
\path[thick, black] ([xshift=3pt]g.east) edge ([xshift=-3pt]f.west);
\path[thick, black] ([xshift=3pt]f.east) edge ([xshift=-3pt]e.west);

\path[thick, black] ([yshift=-3pt]j.south) edge ([yshift=3pt]h.north);
\path[thick, black] ([yshift=3pt]f.north) edge ([yshift=-3pt]c.south);
\end{tikzpicture}
\caption{The Brill-Noether graph $BN((2,2))$}
\label{fig:bngraph}
\end{figure}

For our intended application, we must count the number of vertices and the number of edges of $BN(\lamu)$. It is relatively straightforward to see that the number of vertices is $(n+1)f^\sigma$ (see Lemma \ref{l:numedges}), but the number of edges is more complicated. Our basic tool in computing the number of edges is a probability distribution on staircase paths in Young diagrams, described by the following two definitions.

\begin{defn}
Let $\lamu$ be a connected skew shape with $n$ boxes.  
A {\em staircase path} in $\lamu$ is a path $s$ from the lower left corner to the upper right corner of $\lamu$ that uses only right-steps and up-steps.  
 A {\em turn} in $s$ is a consecutive sequence of steps (right,up) or (up,right) with the property that both steps border a common box of $\lamu$.  
 These turns will be called left and right turns in $s$, respectively.  
\end{defn}

We emphasize that we do not consider a change of direction in $s$ to be a turn unless both of the steps in question border a common box in $\lamu$.  For example, the staircase path in Figure~\ref{fig:staircasepath} has three left turns and one right turn, as shown.  If both steps of a turn border box $(i,j)\in\lamu$, we will say that the turn lies in box $(i,j)$ for short; and we will also say that it lies in the $i^{th}$ row and $j^{th}$ column of $\lamu$.

\begin{figure}
\begin{tikzpicture}[inner sep=0in,outer sep=0in]
\node (s) {\begin{varwidth}{5cm}{
\begin{ytableau}
1 & 2 &3&5&6\\
4 & 7&9&10\\
8 & 11&12
\end{ytableau}}\end{varwidth}};
\draw[line width=2.5pt,blue] ([yshift=0.01cm]s.south west)--([xshift=0.6cm,yshift=0.01cm]s.south west)--([xshift=.6cm,yshift=.6cm]s.south west)--([xshift=0.9cm, yshift=-1.2cm]s.north)--([xshift=0.9cm, yshift=-.6cm]s.north)--([xshift=1.49cm, yshift=-.6cm]s.north)--([yshift=-0.01cm]s.north east);
\draw[fill=blue] ([xshift=0.6cm,yshift=0.01cm]s.south west) circle (3pt);
\draw[ultra thick, color=blue] ([xshift=0.6cm,yshift=0.6cm]s.south west) circle (4pt);
\draw[fill=blue] ([yshift=-0.6cm]s.north east) circle (3pt);
\draw[fill=blue] ([xshift=0.9cm, yshift=-1.2cm]s.north) circle (3pt);
\end{tikzpicture}
\caption{A staircase path in a standard Young tableau of shape $(5,4,3).$ The left and right turns are indicated with solid and open dots respectively.}
\label{fig:staircasepath}
\end{figure}
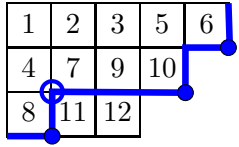

Now, for any $T\in sSYT(\lamu)$ and $m\in \{1,\ldots,n+1\}$, the pair $(T,m)$ naturally defines a staircase path in $\lamu$ that we will denote $s(T,m)$.  
Namely, $s$ is the unique staircase path which divides the entries $<m$ from the entries $\ge m$. 
For example, in Figure~\ref{fig:staircasepath}, the tableau $T$ and $m=11$ determine the staircase path shown.
This allows us to define the Brill-Noether probability distribution on staircase paths as follows.

\begin{defn}\label{d:u}
Let $\lamu$ be a connected skew shape with $n$ boxes. We let $\mu_{BN}$ denote the probability distribution on staircase paths in $\lamu$ obtained by 
\begin{itemize}
\item picking $T\in sSYT(\lamu)$ uniformly at random, 
\item picking $m\in \{1,\ldots,n+1\}$ uniformly at random,
\end{itemize}
and choosing the path $s(T,m)$.  We write $E_\sigma$ for the expected number (i.e.~average number) of turns in a staircase path in $\sigma$ chosen according to the distribution $\mu_{BN}$.  
\end{defn}
\noindent Thus, a staircase path appears in $\mu_{BN}$ with probability proportional to the number of standard fillings of $\lamu$ with which it is compatible, i.e.~with which it divides smaller entries from larger ones.  Note that $\mu_{BN}$ is completely different from the uniform distribution on staircase paths in $\lamu$.  

\begin{defn}\label{d:compression}
Let $T$ be an almost-standard skew Young tableau of shape $\lamu$ of size $n$. 
Let $m \in\{1,\ldots,n+1\}$ be the unique label not appearing in $T$.  
The {\em compression} of $T$, denoted $c(T)$, is the skew standard Young tableau of shape $\lamu$ obtained from $T$ by decrementing each entry of $T$ greater than $m$.
\end{defn}

\noindent Figure~\ref{fig:compression} shows an example.  Note that compression is a map $aSYT(\lamu)\rightarrow SYT(\lamu)$ all of whose fibers have size $n+1$.  

\begin{figure}
\begin{tikzpicture}[inner sep=0in,outer sep=0in]
\node (s) {\begin{varwidth}{5cm}{
\begin{ytableau}
\none & \none & 1 \\ 
\none & 3 & 6\\
4 & 5 
\end{ytableau} \qquad
\begin{ytableau}
\none & \none & 1 \\ 
\none & 2 & 5\\
3 & 4 
\end{ytableau}}\end{varwidth}};
\end{tikzpicture}
\caption{The compression of the almost-standard tableau shown on the left is the standard tableau shown on the right.}
\label{fig:compression}
\end{figure}
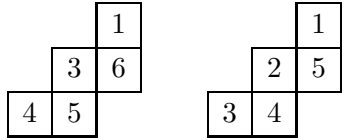

 We may now express the number of edges in $BN(\sigma)$ in terms of the expected number of turns in a staircase path chosen from $\mu_{BN}$.
 
 \begin{lemma}\label{l:numedges}
Suppose $\lamu$ is a connected skew shape with $n$ boxes. The number of vertices in the graph $BN(\lamu)$ is $$(n+1)f^\sigma.$$

The number of edges in the graph $BN(\lamu)$ is 
$$\frac{1}{2}\, (n+1) \, f^{\lamu} E_{\lamu}.$$
\end{lemma}

\begin{proof}
First, consider the map
$$aSYT(\lamu) \longrightarrow SYT(\lamu)\times \{1,\!\ldots\!,n\!+\!1\}$$
sending $T' \in aSYT(\lamu)$  to $(c(T'), m)$.
This map is bijective: given $(T,m)$, the staircase path $s = s(T,m)$ divides the boxes of $\lamu$ into two parts, and incrementing the entries of $T$ southeast of $s$ yields the aSYT that was sent to $(T,m)$.  Thus $BN(\lamu)$ has $(n+1)\, f^\lamu$ vertices, as claimed.

Now suppose $T'$ is a vertex of $BN(\lamu)$, i.e.~an almost-standard skew Young tableau of shape $\lamu$; we wish to compute the degree of $T'$ in $BN(\lamu)$.
Let $m \in \{1,\ldots,n+1\}$ be the number missing from $T'$. 
The vertices in $BN(\lamu)$ adjacent to $T'$ correspond to those aSYT which are obtained from $T'$ by replacing one of the entries of $T'$ with the missing number $m$, so that the result is again an aSYT.  
The key observation is that the entries of $T'$ which may be legally replaced by $m$ correspond precisely to the turns in $s(c(T'), m)$. 
Specifically, a right turn (resp.~left turn) in box $(k,l)$ indicates that the entry in $(k,l)$ is greater than $m$ (resp.~less than $m$) and may be replaced by $m$ so that the rows and columns are still strictly increasing. 
So the degree of $T'$ in $BN(\lamu)$ is the number of turns in the path $s(c(T'), m)$, and we conclude
\begin{eqnarray*}
|E(BN(\lamu))| &=& \frac{1}{2} \, |V(BN(\lamu))| \cdot (\text{average degree of a vertex})\\
&=& \frac{1}{2}\, (n+1) \, f^{\lamu} E_{\lamu}
\end{eqnarray*}
as desired.
\end{proof}

In other words, to count edges in $BN(\sigma)$, we need to compute the expected number of turns in a staircase path in $\sigma$ chosen according to $\mu_{BN}$.   We are able to do this for {\em any} skew shape, in Theorem~\ref{t:masterformula} below.  Before stating this theorem, we fix the following notation. 

\begin{defn}\label{d:sigmai}Let $\sigma=\lambda\setminus\mu$ be a skew shape with $k$ rows.  
For each $i=1,\ldots,k$, we write $\sigma^i$ for the shape obtained from $\sigma$ by adding a box to  row  $i$ on the right, assuming that the result is again a skew shape (i.e. if $\lambda_i < \lambda_{i-1}$). 
 As before,  write $f^{\sigma^i}$ for the number of standard fillings of $\sigma^i$.  By convention, we set $f^{\sigma^i}\!=0$ if $\sigma^i$ is not a skew shape.  
Similarly, we write ${}^i\sigma$ for the shape obtained by adding a box on the left in row $i$, and we define $f^{{}^i\sigma}$ analogously. 
\end{defn}
 
\begin{thm}\label{t:masterformula}
Let $\sigma$ be any connected skew shape, with $n$ boxes and $k$ rows. 
 Then the expected number of turns in a staircase path in $\sigma$ chosen according to the probability distribution $\mu_{\text{BN}}$ is
\begin{equation}\label{eq:masterformula}
E_\sigma = 
2\left(k+ \sum_{i=1}^{k} \frac{k\!-\!i}{n\!+\!1}\cdot\frac{f^{{}^i\sigma}}{f^\sigma} - \sum_{i=1}^{k} \frac{k\!+\!1\!-\!i}{n\!+\!1}\cdot\frac{f^{\sigma^i}}{f^\sigma}\right).
\end{equation}
Moreover, the number of edges in the graph $BN(\sigma)$ is
\begin{equation}
k(n+1)f^\sigma + \sum_{i=1}^{k} (k\!-\!i)\cdot f^{{}^i\sigma} - \sum_{i=1}^{k} (k\!+1\!-\!i)\cdot f^{\sigma^i}.
\end{equation}
\end{thm}

The proof of Theorem~\ref{t:masterformula} relies on two key lemmas of independent interest.  We will state and prove these lemmas, and then return to prove the theorem.   

First, we show that left turns happen just as often as right turns at any given box $(i,j)$.  The correspondence is not immediate; rather, the proof makes surprising use of the structure of the Brill-Noether graph.

\begin{lemma}\label{l:involution}
Let $\lamu$ be a connected skew shape with $n$ boxes, and choose a staircase path $s$ in $\lamu$ according to the distribution $\mu_{BN}$.  Let $(i,j)\in \lamu$ be any box.  Then 
$$\PP(s \text{ has a right turn at box }(i,j)) = \PP(s \text{ has a left turn at box }(i,j)).$$
\end{lemma}
\begin{proof}
First, by a {\em half-edge} of a graph we mean a pair $(v, e) \in V(G) \times E(G)$ such that $e$ is incident to $v$.  
There is an obvious involution $\iota$ on half-edges of any graph, sending $(v,e=vw)$ to   $(w,e)$.  
Now, by the proof of Lemma~\ref{l:numedges}, we see that the vertices of $BN(\lamu)$ are in bijection with pairs $(T,m)\in sSYT(\lamu)\times\{1,\ldots,n+1\}$.  
Furthermore, the half-edges at the vertex corresponding to $(T,m)$ are in bijection with the turns in the staircase path $s(T,m)$. We will write $(T,m,\tau)$ to denote the half-edge corresponding to a turn $\tau$ in the staircase path corresponding to $(T,m)$.

Suppose $\tau$ is a right turn in the path $s(T,m)$ at box $(i,j)$, and consider the half-edge $(T,m,\tau)$.  Let $(T',m',\tau') = \iota(T,m,\tau)$. 
 The aSYT corresponding to $(T',m')$ is obtained from the one corresponding to $(T,m)$ by changing the entry $m'$, located in box $(i,j)$, to $m$.
Furthermore, the fact that $\tau$ was a right turn means precisely that $m'>m$.  
Now consider the half-edge $(T',m',\tau')$. Applying the same argument, we see that $\tau'$ must be a {\em left} turn in box $(i,j)$ in the path $s(T',m')$.  

Summarizing, if a half-edge corresponds to a right turn (resp.~left turn) at box $(i,j)$, then the half-edge with which it is paired corresponds to a left turn (resp.~right turn) at box $(i,j)$.  
We have thus exhibited a bijection
\begin{eqnarray*}
&&\{(T,m,\tau)~|~\tau \text{ is a right turn of $s(T,m)$ at box $(i,j)$}\}\\
&\cong&\{(T',m',\tau')~|~\tau' \text{ is a left turn of $s(T',m')$ at box $(i,j)$}\}
\end{eqnarray*}
and this bijection proves the lemma.
\end{proof}

\noindent See Figure~\ref{fig:bijection} for an example of the bijection in the proof of Lemma~\ref{l:involution}.

\begin{figure}
\begin{tikzpicture}[inner sep=0in,outer sep=0in]
\begin{scope}
\node (s) {\begin{varwidth}{5cm}{
\begin{ytableau}
1 & 2 & 4 & 5\\
3 & 7&9&11\\
6 & 8 & 10 &12
\end{ytableau}}\end{varwidth}};
\draw[line width=2.5pt,blue] ([yshift=0.01cm]s.south west)--([xshift=1.2cm,yshift=0.01cm]s.south west)--([xshift=1.2cm,yshift=1.2cm]s.south west)--([yshift=-.6cm]s.north east)--(s.north east);
\draw[fill=blue] ([xshift=0cm,yshift=-.6cm]s.north east) circle (3pt);
\end{scope}
\begin{scope}[shift={(4,0)}]
\node (s) {\begin{varwidth}{5cm}{
\begin{ytableau}
1 & 2 & 4 & 8\\
3 & 6&9&11\\
5 & 7 & 10 &12
\end{ytableau}}\end{varwidth}};
\draw[line width=2.5pt,blue] ([yshift=0.01cm]s.south west)--([yshift=0.6cm]s.south west)--([xshift=0.6cm,yshift=0.6cm]s.south west)--([xshift=.6cm,yshift=1.2cm]s.south west)--([xshift=-.6cm, yshift=-0.6cm]s.north east)--([xshift=-0.6cm]s.north east)--([yshift=-0.01cm]s.north east);
\draw[fill=blue] ([xshift=-.6cm,yshift=0cm]s.north east) circle (3pt);
\end{scope}
\end{tikzpicture}
\caption{Lemma~\ref{l:involution} establishes a bijection between pairs $(T,m)$ such that $s(T,m)$ has a right turn at $(i,j)$ and pairs $(T',m')$
 such that $s(T',m')$ has a left turn at $(i,j)$.  
 This figure illustrates one instance of this bijection, when $(i,j) = (1,4)$.  Here, $T$ and $T'$ are as shown and $m=9, m'=5$.\medskip}
\label{fig:bijection}
\end{figure}
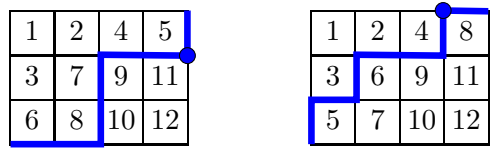

The second key lemma can be stated completely independently of staircase paths and Brill-Noether graphs.  It addresses, in one case, the very natural question: what is the expected value of a given box in a randomly chosen standard Young tableau of fixed shape?   

\begin{lemma}\label{p:upperright}
Let $\sigma$ be a skew shape with $n$ boxes. 
\begin{enumerate}[(i)]
\item
The expected value of the last entry of the first row of a uniformly chosen skew standard Young tableau of shape $\sigma$ is
$$n+1-f^{\sigma^1}/f^{\sigma},$$
where $\sigma^1$ denotes the skew shape obtained from $\sigma$ by adding a box to the first row on the right.  

\item More generally, suppose $b=(i,j)$ is the last box of the $i^{th}$ row of $\sigma$, and suppose the box 
$c=(i-1,j+1)$ due northeast of $b$ also lies in $\sigma.$ 
 Then the expected value of the maximum value in boxes $b$ and $c$ of a uniformly chosen skew standard Young tableau of shape $\sigma$ is
$$n+1-f^{\sigma^i}/f^{\sigma},$$
where $\sigma^i$ denotes the skew shape obtained from $\sigma$ by adding a box to the $i^{th}$ row on the right.  

\end{enumerate}
\end{lemma}

\begin{remark}
In the case that $\sigma$ is non-skew, the quantities in Lemma~\ref{p:upperright} can be calculated explicitly via the hook-length formula, 
using the fact that the only hook lengths that change upon adding a box to $\sigma$ are the ones in the same row or the same column as the new box. 
 For example, it follows that the expected value of the upper right corner box of a uniformly chosen standard Young tableau of shape $\sigma$ is 
\begin{equation}\label{eq:upperright}
(n+1)\left(1-\prod_{j=1}^{\lamu_1} \frac{\lamu_1 + \lamu_j^* -j}{\lamu_1 + \lamu_j^* -j + 1} \right).
\end{equation}

We also remark that by rotating $180$-degrees and replacing each number $i$ with $n+1-i$, we immediately obtain the analogous result on the left border of $\sigma.$ 
 Namely, if $b$ is the leftmost box of the $i^{th}$ row if $\sigma$ and $c\in\sigma$ is the box due southwest of it, then we have
$$\displaystyle \mathbb{E}\left(\min (T(b),T(c)) \right) = f^{{}^i\sigma}/f^\sigma,$$
for a uniformly chosen $T\in sSYT(\sigma)$.
\end{remark}
\begin{proof}[Proof of Lemma~\ref{p:upperright}]

Part (i) is simply a degenerate special case of part (ii), which we now prove.  Consider the shape $\sigma^i$ obtained by adding a box on the right of the $i^{th}$ row of $\sigma$.  Note $\sigma^i$ is again a skew shape, by the assumption that the box $(i-1,j+1)$ also lies in~$\sigma$.  

Consider the map
$$\Phi\colon sSYT(\sigma^i)\rightarrow sSYT(\sigma)$$
defined as follows: given $T'\in sSYT(\sigma^i)$, erase the last box in the $i^{th}$ row of $T'$, and take the compression of the resulting almost-standard skew Young tableau. 

For example, if $\sigma = (5,2) \setminus (1)$ and $i=2$, then $\Phi$ sends each of
$$ \young (:1267,345)\qquad \young(:1257,346)\qquad \young (:1256,347) $$
\smallskip
to the following tableau.
\smallskip
$$\young(:1256,34)$$
\smallskip
Now given $T\in sSYT(\sigma)$, we have $$|\Phi^{-1}(T)| = n+1-\max (T(b), T(c)),$$ since an element in 
$\Phi^{-1}(T)$ is obtained from $T$ by picking any $\alpha \in \{\max (T(b), T(c))\!+\!1, \ldots, n+1\}$, incrementing all entries of $T$ that are at least $\alpha$, and then writing $\alpha$ in the extra box.  This is illustrated in the example above.

Now double-counting the size of the domain, we have
$$f^{\sigma^i} = \sum_{T\in sSYT(\sigma)}\!(n+1-\max (T(b), T(c))) = f^{\sigma} \cdot (n+1 - \E(\max (T(b), T(c)))).$$
We conclude $$\E \left( \max (T(b), T(c)) \right) = n+1-f^{\sigma^i}/f^{\sigma}.$$  

\end{proof}

Lemma~\ref{p:upperright} is new as far as we know.  In general, it is very natural to ask for an explicit  formula, given a partition $\lambda$ of $n$, for the expected value of any given box of a uniformly chosen standard tableau of shape $\lambda$.  This is likely to be difficult to achieve in general.  A formula for the expected value of box $(2,1)$ for any Young diagram was given recently in \cite{rn}. Apart from their result, our Lemma~\ref{p:upperright}(i), and trivial cases like the box $(1,1)$, we do not know of any other results along these lines.

Now we turn to the proof of Theorem~\ref{t:masterformula}.

\begin{proof}[Proof of Theorem~\ref{t:masterformula}]
Let $\sigma$ have $n$ boxes and $k$ rows. 
 If $s$ is a staircase path in $\sigma$, we will say that $s$ has a right turn (respectively left turn) in row $i$ if it has a right turn (respectively left turn) in box $(i,j)$ for some $j$.  
 For $i=1,\ldots,k$, write $R_i$ and $L_i$ for the event that $s$ has a right turn in row $i$, respectively a left turn in row $i$. 
  Clearly, any staircase path has either 0 or 1 right turns (respectively, left turns) in a given row, so
\begin{equation}\label{eq:esigma}
\begin{aligned}
E_\sigma &= \PP(R_1) + \cdots + \PP(R_k) + \PP(L_1)+\cdots+\PP(L_k) \\
&= 2(\PP(R_1) + \cdots + \PP(R_k)),
\end{aligned}
\end{equation}
where the latter equality follows from Lemma~\ref{l:involution}.
So we wish to calculate $\PP(R_i)$ for each $i$.  We will do this by expressing $\PP(R_{i+1})$ in terms of $\PP(R_i)$ and then summing up; we'll take advantage of the fact that $\PP(R_i)=\PP(L_i)$ throughout the argument.

First, we note
\begin{equation}
\PP(R_{i+1})\,+\, \PP(L_i \text{ and not } R_{i+1})\, =\, \PP(R_{i+1} \text{ and not } L_i) \,+ \,\PP(L_i),\end{equation}
since both are equal to $\PP(R_{i+1} \text{ or } L_i).$  Using $\PP(L_i) = \PP(R_i),$ we get
\begin{equation}\label{eq:recurse}
\PP(R_{i+1}) \,=\,  \PP(R_i) \,+ \,\PP(R_{i+1} \text{ and not } L_i) \,-\,\PP(L_i \text{ and not } R_{i+1}).\end{equation}

The next claim is then the key step to relating $\PP(R_i)$ and $\PP(R_{i+1})$.
\begin{claim}\label{c:rowdependence}
For each $i$, we have
\begin{eqnarray}
\PP(L_i \text{ and not } R_{i+1}) &=& \frac{1}{n\!+\!1}\frac{f^{\sigma^{i+1}}}{f^\sigma}\label{eq:li}\\[.1cm] 
\PP(R_{i+1} \text{ and not } L_i) &=& \frac{1}{n\!+\!1}\frac{f^{{}^i\sigma}}{f^\sigma}.\label{eq:ri+1}
\end{eqnarray}
In particular, 
\begin{equation}\label{eq:r1}
\PP(\text{not }R_1) = \frac{1}{n\!+\!1}\frac{f^{\sigma^{1}}}{f^\sigma} \qquad \text{and}\qquad \PP(\text{not }L_k) = \frac{1}{n\!+\!1}\frac{f^{{}^k\sigma}}{f^\sigma}.
\end{equation}
\end{claim}

\begin{proof}[Proof of Claim~\ref{c:rowdependence}]
The statements in~\eqref{eq:r1} are degenerate special cases of~\eqref{eq:li} and~\eqref{eq:ri+1}.  Furthermore~\eqref{eq:ri+1} is obtained directly from~\eqref{eq:li} by applying a 180-degree rotation of $\sigma$.  So it remains to prove~\eqref{eq:li}.

Suppose $s$ is a staircase path in $\sigma$ that has a left turn in row $i$.   
As usual, regard $s$ as a path starting from the lower-left corner and ending at the upper-right corner.  
Now, immediately prior to the left turn in row $i$, it must be the case that $s$ is traveling to the right along a horizontal segment lying between rows $i$ and $i+1$. 
 Immediately before that horizontal segment, $s$ must of course be traveling up. 

Now, if this up-step occurs anywhere but the right border of row $i$, then $s$ has a right turn in row $i+1$. In particular, if the $i^{th}$ row of $\sigma$ extends to the right only as far as the $(i+1)^{st}$ row, then {\em any} staircase path with a left turn in row $i$ also has a right turn in row $i+1$.  
Then the quantity on the left in~\eqref{eq:li} is zero; but so is $f^{\sigma^{i+1}}$ on the right, by our convention, which proves~\eqref{eq:li} in this case.  

So we may assume that the $i^{th}$ row of $\sigma$ extends further to the right than the $(i+1)^{st}$ row.
In this case, $s$ fails to have a right turn in row $i+1$ if and only if the up-step traversing row $i+1$ occurs on the right border of row $i+1$. 
 Let us call this up-step, followed by the next right-step, the {\em outward corner} $X_{i+1}$.  See Figure~\ref{fig:xi+1}.
\begin{figure}
\begin{tikzpicture}[inner sep=0in,outer sep=0in]
\node (s) {\begin{varwidth}{5cm}{
\begin{ytableau}
 \empty & \empty & \empty & \empty & c &\empty \\
 \empty & \empty & \empty& b\\
 \empty & \empty &\empty  
\end{ytableau}}
\end{varwidth}};
\draw[line width=2.5pt,blue] ([yshift=0.01cm]s.south west)--([xshift=0.6cm,yshift=0.01cm]s.south west)--
([xshift=0.6cm, yshift=0.6cm]s.south west)--
([xshift=2.4cm, yshift=0.6cm]s.south west)--
([xshift=2.4cm, yshift=1.2cm]s.south west)--
([xshift=3cm, yshift=1.2cm]s.south west)--
([xshift=3cm, yshift=1.8cm]s.south west)--
([xshift=3.6cm, yshift=1.8cm]s.south west);
\draw[line width=2.5pt,blue] ([xshift=2.42cm, yshift=1.05cm]s.south west)--
([xshift=2.55cm, yshift=1.2cm]s.south west);
\node at ([xshift=2.95cm, yshift=0.85cm]s.south west) {$X_{i+1}$};
\end{tikzpicture}
\caption{Illustration for the proof of Claim~\ref{c:rowdependence}.  In this example $\sigma=(6,4,3)$ and $i=1$.  Note that there is a left turn in row $i$ but no right turn in row $i+1$, because the path uses the outward corner marked $X_{i+1}$.}
\label{fig:xi+1}
\end{figure}
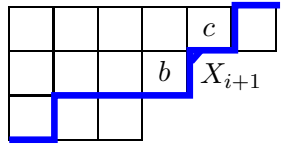

Then we have shown
$$\PP(L_i \text{ and not } R_{i+1}) = \PP(\text{$X_{i+1}$ is used}).$$
Now, pick $T\in sSYT(\sigma)$ uniformly at random, and consider the $n+1$ staircase paths
$$s(T,1), \ldots,s(T,n+1).$$
Let $b$ and $c$ denote the boxes of $\sigma$ immediately to the left of $X_{i+1}$ and immediately above $X_{i+1}$, respectively.  
Notice that a staircase path $s(T,m)$ uses the outward corner $X_{i+1}$ precisely when the tableau lying above that path contains both $b$ and $c$. 
 In other words, exactly $n+1 - \max (T(b), T(c))$ of the $n+1$ staircase paths defined by $T$ use the outward corner $X_{i+1}$.  Thus, by linearity, it follows that
\begin{eqnarray*}
\PP(\text{$X_{i+1}$ is used}) &=& 1 - \frac{\E \max (T(b), T(c))}{n+1}\\[.1cm]
&=& \frac{1}{n\!+\!1}\frac{f^{\sigma^{i+1}}}{f^\sigma} 
\end{eqnarray*}
where the last equality is by Lemma~\ref{p:upperright}.  This proves Claim~\ref{c:rowdependence}.
\end{proof}

\noindent Returning to the proof of Theorem~\ref{t:masterformula}, it follows from Equation~\eqref{eq:recurse} and Claim~\ref{c:rowdependence}  that
\begin{equation}\label{eq:relation}
\PP(R_{i+1}) = \PP(R_i) + \frac{1}{n\!+\!1}\frac{f^{{}^i\sigma}}{f^\sigma} -\frac{1}{n\!+\!1}\frac{f^{\sigma^{i+1}}}{f^\sigma}.
\end{equation}
Furthermore, by Claim~\ref{c:rowdependence}, we have $$\PP(R_1) = 1-\frac{1}{n\!+\!1}\frac{f^{\sigma^{1}}}{f^\sigma}.$$

Applying~\eqref{eq:relation} repeatedly to the expression $E_\sigma=2(\PP(R_1) + \cdots + \PP(R_k))$ from~\eqref{eq:esigma}, we get exactly the quantity in the first part of the theorem statement.  Note in particular that each term $-\frac{1}{n+1}\frac{f^{\sigma^i}}{f^\sigma}$ appears a total of $k+1-i$ times in the sum, and each term $\frac{1}{n+1}\frac{f^{{}^i\sigma}}{f^\sigma}$ appears a total of $k-i$ times.
This proves the first part of Theorem~\ref{t:masterformula}, and the second part follows from Lemma~\ref{l:numedges}.

\end{proof}

\begin{remark}
There is an implicit symmetry in the expression~\eqref{eq:masterformula} which, when unraveled, gives a nice combinatorial identity. 
 Replace $\sigma$ by a 180-degree rotation of $\sigma$ and apply the same expression.
 Equating the two expressions so obtained and simplifying gives the very simple identity:
\begin{cor}\label{c:strangeidentity}
For any skew shape $\sigma$ with  $k$ rows, we have
\begin{equation}\label{eq:strangeidentity}
\sum_{i=1}^k f^{{}^i\sigma} = \sum_{i=1}^k f^{\sigma^i}.
\end{equation}
\end{cor}
This is a curious identity on counts of skew tableaux, derived from our analysis of turns in staircase paths, that is a new result as far as we know.

Applying~\eqref{eq:masterformula} to the conjugate of $\sigma$, interchanging rows and columns, gives another identity on counts of standard fillings of skew shapes, but this is presumably more complicated.
\end{remark}

Theorem~\ref{t:masterformula} computes the expected number of turns of a staircase path in $\sigma$ chosen from $\mu_{BN}$ in any shape $\sigma.$ 
 For reference, we will work out two useful examples.  

\begin{cor}\label{t:harmonic}
Let $a$ and $b$ be two positive integers. Let $\lam = (b,\ldots,b)$, where $b$ occurs $a$ times, i.e.~the Young diagram of $\lam$ is an $a\times b$ rectangle.  Then
$$E_\lam= \frac{2ab}{a+b},$$
the harmonic mean of $a$ and $b$.
Moreover,  
the number of edges in the graph $BN(\lam)$ is 
$$\frac{ab}{a+b}\cdot (ab+1)! \prod_{j=0}^{b-1} \frac{j!}{(a+j)!}.$$
\end{cor}

\begin{proof}

The expression~\eqref{eq:masterformula} simplifies to
$$E_\lambda = 2\left(a - \frac{a}{ab\!+\!1}\frac{f^{\lambda^1}}{f^\lambda}\right).$$
By the hook-length formula applied to $\lambda^1$ and $\lambda$, and noting the telescoping cancellation, we have 
$$\frac{f^{\lambda^1}}{f^\lambda} = \frac{a(ab+1)}{a+b},$$
and we derive
$$E_\lambda = \frac{2ab}{a+b}$$
which proves the first statement.
To prove the second statement, we use the first statement and the formula in Lemma~\ref{l:numedges}.  
By the hook-length formula, we have
\begin{eqnarray*}
f^\lam &=& (ab)! \frac{(b-1)!(b-2)!\cdots 0!}{(a+b-1)!(a+b-2)!\cdots a!}\\
&=& (ab)!\cdot \prod_{j=0}^{b-1} \frac{j!}{(a+j)!}.
\end{eqnarray*}
Therefore, the number of edges in $BN(\lam)$ is 
$$\frac{ab}{a+b}\cdot (ab+1)! \prod_{j=0}^{b-1} \frac{j!}{(a+j)!}.$$
\end{proof}

\begin{remark}
It is remarkable that the expected number of turns in Corollary~\ref{t:harmonic} is exactly the same 
as the expected number of turns of a {\em uniformly chosen} staircase path  in an $a\times b$ box, as the following easy calculation shows. 
 Regard a staircase path $s$ as a sequence of $a$ entries $U$ and $b$ entries $R$, corresponding to the $a$ steps up and $b$ steps to the right in $s$.  
 Then each of the $a+b-1$ pairs of consecutive entries in this sequence determine a turn if they are $(U,R)$ or $(R,U)$, and the probability of each of these events is $\tfrac{a}{a+b}\cdot \tfrac{b}{a+b-1}$. 
  Summing, the expected number of turns in $s$, chosen uniformly from the set of staircase paths, is $2ab/(a+b)$.  
  Of course, the two probability distributions themselves are completely different.  
\end{remark}

We now present a more general case in which~\eqref{eq:masterformula} still has a simple expression. 
 This corollary computes the genera of Brill-Noether curves in a wide variety of new cases.
For example, it pertains to the case of arbitrary $g,r,$ and $d$ and one point of simple ramification.
\begin{cor} \label{c:tworectangles}
Let $\sigma$ be any skew shape 
obtained from some two-row shape by repeating the first row $k_1$ times and the second row $k_2$ times, for any $k_1$ and $k_2$.  Let $k=k_1 + k_2$.  Then 
$$E_\sigma=2k_1(1-\frac{1}{n\!+\!1}\frac{f^{\sigma^1}}{f^\sigma})+2k_2(1-\frac{1}{n\!+\!1}\frac{f^{{}^k\sigma}}{f^\sigma}),$$
and the number of edges in the graph $BN(\sigma)$ is
$$k(n+1)f^\sigma - f^{\sigma^1} + f^{{}^k\sigma}.$$ 
\end{cor}

\begin{proof}
This follows directly from Theorem~\ref{t:masterformula} and the identity~\eqref{eq:strangeidentity}.  Alternatively, it can be calculated using the method in the proof of Theorem~\ref{t:masterformula} twice, working down from the first row to calculate the number of turns in the first $k_1$ rows, and working up from the last row to calculate the number of turns in the last $k_2$ rows.
\end{proof}

\section{Pontableaux and valid sequences}\label{sec:ponttab}
	
In this section we introduce two combinatorial notions: \textit{pontableaux} and \textit{valid sequences}. Each notion depends on a choice of data $g,r,d,\alpha,\beta$, and is used to enumerate the components of $\Grdab(X,p,q)$ when $X$ is a chain of elliptic curves. Both notions encode exactly the same information in the case that the {\em adjusted Brill-Noether number} $\rho(g,r,d,\alpha,\beta)$, defined in~\eqref{eq:adjustedrho}, is equal to $1$.  

We remark that our definition of pontableaux will be made only in the case $\rho(g,r,d,\alpha,\beta)=1$. This is merely to keep the notation as simple as possible and to focus on the intended application. Our definition of valid sequences works for all values of $\rho$.

     We start by defining pontableaux (by way of an intermediate definition of {\em pretableaux}).
  We also  define an adjacency relation on pontableaux.  
  We have chosen the word ``pontableau'' because these objects include tableaux as well as new objects which form bridges (in French, \textit{ponts}) 
  between them in an augmented version of the Brill-Noether graph (see Figure \ref{fig:subbngraph}). 
  
Recall that given two partitions $\lam$ and $\mu$ such that $\mu_i\,{\leq}\, \lambda_i$ for all $i$,
  the shape resulting from removing $\mu$ from $\lambda$ is called a \emph{skew Young diagram} or \emph{skew shape} (see \ref{d:skewYT}).  We note that beginning in this section, we will refer to the topmost row of a skew shape as row $0$ instead of row $1$, and so on, in order to agree with the notational conventions on vanishing orders in subsequent sections.

 \begin{defn} \label{def:skewtab}
 Given $g, r, d$ non-negative integers with $g-d+r \geq 0$, $\alpha =(\alpha_0, \dots,\alpha_r)$ a non-decreasing  and  $\beta= (\beta_0,\dots,\beta_r)$ a non-increasing $(r+1)$-tuple of integers,
 construct a skew shape $\sigma(g, r, d, \alpha, \beta)$  as follows: 
 given the rectangular Young diagram associated to the partition $(g-d+r, g-d+r,\dots,g-d+r)$ of $(r+1)(g-d+r)$, 
 lengthen the $i^\text{th}$ row by attaching $\alpha_i$ boxes to the left and $\beta_{i}$ boxes to the right, for each $i=0,\ldots,r$ (see Figure~\ref{fig:grdab}).
\end{defn}

We will use the following convention to refer to the boxes in the shape $\sigma$: the upper-leftmost box of the $(g-d+r)\times (r+1)$ rectangle will be $(0,0) $
and the $x$ and $y$ coordinates will increase to the right and down, respectively. In particular, the box $(0,0)$ of $\sigma$ may not be the upper-leftmost box of $\sigma$, if $\alpha_0>0$. 

For any skew shape $\sigma$, recall the definitions of $\sigma^a, {}^a\sigma, f^{\sigma^a},$ and $f^{{}^a\sigma}$ from Definition~\ref{d:sigmai}.  

\begin{defn}\label{d:pretableau}
Fix a skew shape $\sigma$ with $n$ boxes.  A {\em pretableau} is a filling of the boxes in $\sigma$, $\sigma^a$, or ${}^a\sigma$ (for some integer $a$) with the symbols 
$1,\ldots,n+1$ (each used exactly once)
and $-\!j$, for one value $j\in \{1,\ldots,n+1\}$, such that:
\begin{enumerate}
\item Counting each symbol $1,\ldots,n\!+\!1$ with weight $+1$ and the symbol $-j$ with weight $-1$, every box in $\sigma$ has total weight $1$, 
while every box outside $\sigma$ has total weight 0.
\item Choosing one positive symbol  from every box produces a filling in which all rows and columns are strictly increasing. 
\item Let $b$ be the box containing the negative symbol $-j$.  Then:
\begin{itemize}
\item If $b$ lies in or to the right of $\sigma$, then there is some symbol $i$ in $b$ with $i\le j$.
\item If $b$ lies in or to the left of $\sigma$, then there is some symbol $i$ in $b$ with $i\ge j$.
\end{itemize}
\end{enumerate} 
\end{defn}

\begin{defn}\label{d:preadjacency}
We say that two pretableaux are {\em adjacent} if they are identical except for replacing a single symbol $-j$ with $-(j\!+\!1)$ or $-(j\!-\!1)$.
\end{defn}

\begin{defn}\label{d:pontableau}
Fix a skew shape $\sigma$.  A {\em pontableau} is an equivalence class of pretableaux of shape $\sigma$ under the relation of moving a pair of opposite labels $j,-j$
 that occupy the same box to any other (allowable) single box.  
\end{defn}

\noindent In this definition, we allow the shape of the pretableau to change under the operation of moving the pair $j,-j$,
 as long as the new shape is still of the form $\sigma$, $\sigma ^a$ or ${}^a\sigma$ for some $a$.

Note that if $j$ and $-j$ occupy the same box of a pretableau $T$, then deleting $j,-j$ entirely produces an almost-standard Young tableau of shape $\sigma$. 
 Thus we may regard the set of pontableaux as the union of the set of almost-standard Young tableaux on $\sigma$
  together with the set of pretableaux in which $-j$ and $j$ appear in distinct boxes.

\begin{ex}\label{d:preequivalence}  The four pretableaux on shape $\sigma = (2,2)$ below form an equivalence class.
\medskip

{
\centering
\begin{tabular}{|c|c|c|}
\cline{1-3}
$\st 1$ & $\st 2$ & $\st -\!3,3$ \\
\cline{1-3}
$\st 4$ & $\st 5$ & \multicolumn{1}{c}{} \\
\cline{1-2}
\end{tabular}
\qquad 
\begin{tabular}{|c|c|c|}
\cline{2-3}
\multicolumn{1}{c|}{} & $\st 1$ & $\st 2$ \\
\cline{1-3}
$\st -\!3,3$ & $\st 4$ & $\st 5$  \\
\cline{1-3}
\end{tabular}
\qquad 
\begin{tabular}{|c|c|}
\hline
$\st 1$ & $\st 2,-\!3,3$ \\
\hline
$\st 4$ & $\st 5$\\
\hline
\end{tabular}
\qquad
\begin{tabular}{|c|c|}
\hline
$\st 1$ & $\st 2$ \\
\hline
$\st -\!3,3,4$ & $\st 5$\\
\hline
\end{tabular}

}
\medskip
\noindent This equivalence class is indexed by the almost-standard tableau
\medskip

{\centering
\begin{tabular}{|c|c|}
\cline{1-2}
$\st 1$ & $\st 2$ \\
\cline{1-2}
$\st 4$ & $\st 5$ \\
\cline{1-2}
\end{tabular}

}
\medskip 

As another example, the 20 pontableaux on a $2\times 2$ square are shown in Figure~\ref{fig:subbngraph}.
\end{ex}

\begin{defn}\label{d:ponadjacency}
We define the {\em augmented Brill-Noether graph} $BN'(\sigma)$ as the graph whose vertices are pontableaux of shape $\sigma$, and where two pontableaux are adjacent in $BN'(\sigma)$ if they have pretableau representatives that are adjacent (see Figures~\ref{fig:subbngraph} and~\ref{fig:adjacency}).
\end{defn}

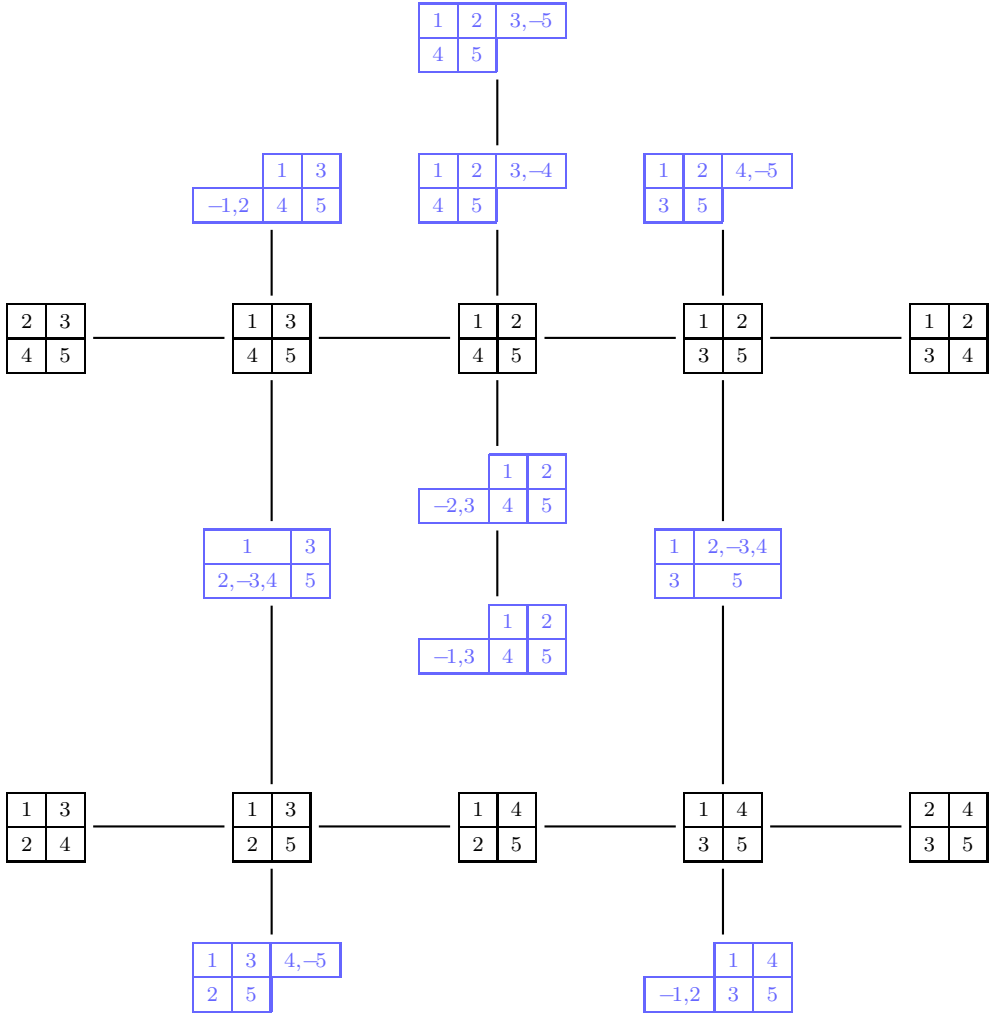
\begin{figure}

\begin{tikzpicture}[inner sep=0in,outer sep=0in]

\node (R1) at (-4,6){\begin{varwidth}{5cm}{
\color{blue!60!white}{
\begin{tabular}{|c|c|c|}
\cline{1-3}
$\st 1$ & $\st 2$ & $\st 3, -\!5$\\
\cline{1-3}
$\st 4$ & $\st 5$ & \multicolumn{1}{|c}{}\\
\cline{1-2}
\end{tabular}
}}\end{varwidth}};

\node (R2) at (-7,4){\begin{varwidth}{5cm}{
\color{blue!60!white}{
\begin{tabular}{|c|c|c|}
\cline{2-3}
\multicolumn{1}{c|}{} & $\st 1$ & $\st 3$\\
\cline{1-3}
$\st -\!1, 2$ & $\st 4$ & $\st 5$\\
\cline{1-3}
\end{tabular}
}}\end{varwidth}};

\node (R3) at (-4,4){\begin{varwidth}{5cm}{
\color{blue!60!white}{
\begin{tabular}{|c|c|c|}
\cline{1-3}
$\st 1$ & $\st 2$ & $\st 3, -\!4$\\
\cline{1-3}
$\st 4$ & $\st 5$ & \multicolumn{1}{|c}{}\\
\cline{1-2}
\end{tabular}
}}\end{varwidth}};

\node (R4) at (-1,4){\begin{varwidth}{5cm}{
\color{blue!60!white}{
\begin{tabular}{|c|c|c|}
\cline{1-3}
$\st 1$ & $\st 2$ & $\st 4, -\!5$\\
\cline{1-3}
$\st 3$ & $\st 5$ & \multicolumn{1}{|c}{}\\
\cline{1-2}
\end{tabular}
}}\end{varwidth}};

\node (E1) at (-10,2){\begin{varwidth}{5cm}{
\begin{tabular}{|c|c|}
\cline{1-2}
$\st 2$ & $\st 3$\\
\cline{1-2}
$\st 4$ & $\st 5$\\
\cline{1-2}
\end{tabular}
}\end{varwidth}};

\node (E2) at (-7,2){\begin{varwidth}{5cm}{
\begin{tabular}{|c|c|}
\cline{1-2}
$\st 1$ & $\st 3$\\
\cline{1-2}
$\st 4$ & $\st 5$\\
\cline{1-2}
\end{tabular}
}\end{varwidth}};

\node (E3) at (-4,2){\begin{varwidth}{5cm}{
\begin{tabular}{|c|c|}
\cline{1-2}
$\st 1$ & $\st 2$\\
\cline{1-2}
$\st 4$ & $\st 5$\\
\cline{1-2}
\end{tabular}
}\end{varwidth}};

\node (E4) at (-1,2){\begin{varwidth}{5cm}{
\begin{tabular}{|c|c|}
\cline{1-2}
$\st 1$ & $\st 2$\\
\cline{1-2}
$\st 3$ & $\st 5$\\
\cline{1-2}
\end{tabular}
}\end{varwidth}};

\node (E5) at (2,2){\begin{varwidth}{5cm}{
\begin{tabular}{|c|c|}
\cline{1-2}
$\st 1$ & $\st 2$\\
\cline{1-2}
$\st 3$ & $\st 4$\\
\cline{1-2}
\end{tabular}
}\end{varwidth}};

\node (R5) at (-4,0){\begin{varwidth}{5cm}{
\color{blue!60!white}{
\begin{tabular}{|c|c|c|}
\cline{2-3}
\multicolumn{1}{c|}{} & $\st 1$ & $\st 2$\\
\cline{1-3}
$\st -\!2, 3$ & $\st 4$ & $\st 5$\\
\cline{1-3}
\end{tabular}
}}\end{varwidth}};

\node (R6) at (-7,-1.0){\begin{varwidth}{5cm}{
\color{blue!60!white}{
\begin{tabular}{|c|c|}
\cline{1-2}
$\st 1$ & $\st 3$\\
\cline{1-2}
$\st 2,-\!3, 4$ & $\st 5$\\
\cline{1-2}
\end{tabular}
}}\end{varwidth}};

\node (R7) at (-4,-2){\begin{varwidth}{5cm}{
\color{blue!60!white}{
\begin{tabular}{|c|c|c|}
\cline{2-3}
\multicolumn{1}{c|}{} & $\st 1$ & $\st 2$\\
\cline{1-3}
$\st -\!1, 3$ & $\st 4$ & $\st 5$\\
\cline{1-3}
\end{tabular}
}}\end{varwidth}};

\node (R8) at (-1,-1.0){\begin{varwidth}{5cm}{
\color{blue!60!white}{
\begin{tabular}{|c|c|}
\cline{1-2}
$\st 1$ & $\st 2, -\!3, 4$\\
\cline{1-2}
$\st 3$ & $\st 5$\\
\cline{1-2}
\end{tabular}
}}\end{varwidth}};

\node (E6) at (-10,-4.5){\begin{varwidth}{5cm}{
\begin{tabular}{|c|c|}
\cline{1-2}
$\st 1$ & $\st 3$\\
\cline{1-2}
$\st 2$ & $\st 4$\\
\cline{1-2}
\end{tabular}
}\end{varwidth}};

\node (E7) at (-7,-4.5){\begin{varwidth}{5cm}{
\begin{tabular}{|c|c|}
\cline{1-2}
$\st 1$ & $\st 3$\\
\cline{1-2}
$\st 2$ & $\st 5$\\
\cline{1-2}
\end{tabular}
}\end{varwidth}};

\node (E8) at (-4,-4.5){\begin{varwidth}{5cm}{
\begin{tabular}{|c|c|}
\cline{1-2}
$\st 1$ & $\st 4$\\
\cline{1-2}
$\st 2$ & $\st 5$\\
\cline{1-2}
\end{tabular}
}\end{varwidth}};

\node (E9) at (-1,-4.5){\begin{varwidth}{5cm}{
\begin{tabular}{|c|c|}
\cline{1-2}
$\st 1$ & $\st 4$\\
\cline{1-2}
$\st 3$ & $\st 5$\\
\cline{1-2}
\end{tabular}
}\end{varwidth}};

\node (E10) at (2,-4.5){\begin{varwidth}{5cm}{
\begin{tabular}{|c|c|}
\cline{1-2}
$\st 2$ & $\st 4$\\
\cline{1-2}
$\st 3$ & $\st 5$\\
\cline{1-2}
\end{tabular}
}\end{varwidth}};

\node (R9) at (-7,-6.5){\begin{varwidth}{5cm}{
\color{blue!60!white}{
\begin{tabular}{|c|c|c|}
\cline{1-3}
$\st 1$ & $\st 3$ & $\st 4, -\!5$\\
\cline{1-3}
$\st 2$ & $\st 5$ & \multicolumn{1}{|c}{}\\
\cline{1-2}
\end{tabular}
}}\end{varwidth}};

\node (R10) at (-1,-6.5){\begin{varwidth}{5cm}{
\color{blue!60!white}{
\begin{tabular}{|c|c|c|}
\cline{2-3}
\multicolumn{1}{c|}{} & $\st 1$ & $\st 4$\\
\cline{1-3}
$\st -\!1, 2$ & $\st 3$ & $\st 5$\\
\cline{1-3}
\end{tabular}
}}\end{varwidth}};

\path[thick, black] ([yshift=-3pt]R1.south) edge ([yshift=3pt]R3.north);
\path[thick, black] ([yshift=-3pt]R3.south) edge ([yshift=3pt]E3.north);
\path[thick, black] ([yshift=-3pt]R2.south) edge ([yshift=3pt]E2.north);
\path[thick, black] ([yshift=-3pt]R4.south) edge ([yshift=3pt]E4.north);
\path[thick, black] ([yshift=-3pt]E2.south) edge ([yshift=3pt]R6.north);
\path[thick, black] ([yshift=-3pt]E3.south) edge ([yshift=3pt]R5.north);
\path[thick, black] ([yshift=-3pt]R5.south) edge ([yshift=3pt]R7.north);
\path[thick, black] ([yshift=-3pt]E4.south) edge ([yshift=3pt]R8.north);
\path[thick, black] ([yshift=-3pt]R6.south) edge ([yshift=3pt]E7.north);
\path[thick, black] ([yshift=-3pt]R8.south) edge ([yshift=3pt]E9.north);
\path[thick, black] ([yshift=-3pt]E7.south) edge ([yshift=3pt]R9.north);
\path[thick, black] ([yshift=-3pt]E9.south) edge ([yshift=3pt]R10.north);

\path[thick, black] ([xshift=3pt]E1.east) edge ([xshift=-3pt]E2.west);
\path[thick, black] ([xshift=3pt]E2.east) edge ([xshift=-3pt]E3.west);
\path[thick, black] ([xshift=3pt]E3.east) edge ([xshift=-3pt]E4.west);
\path[thick, black] ([xshift=3pt]E4.east) edge ([xshift=-3pt]E5.west);

\path[thick, black] ([xshift=3pt]E6.east) edge ([xshift=-3pt]E7.west);
\path[thick, black] ([xshift=3pt]E7.east) edge ([xshift=-3pt]E8.west);
\path[thick, black] ([xshift=3pt]E8.east) edge ([xshift=-3pt]E9.west);
\path[thick, black] ([xshift=3pt]E9.east) edge ([xshift=-3pt]E10.west);
\end{tikzpicture}
\caption{The augmented Brill-Noether graph $BN'((2,2))$}
\label{fig:subbngraph}
\end{figure}

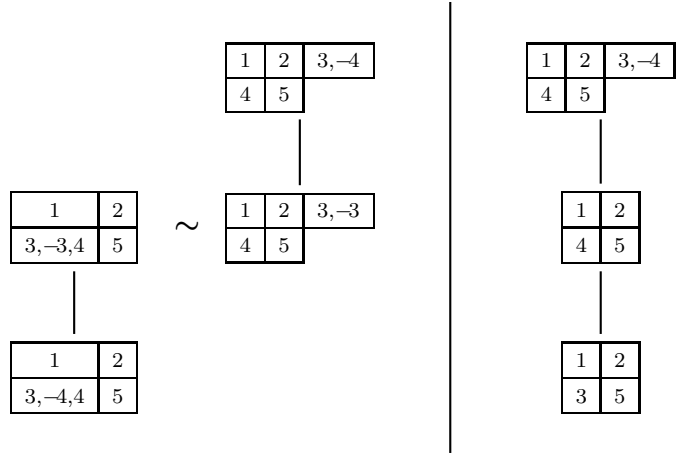
\begin{figure}
\begin{tikzpicture}[inner sep=0in,outer sep=0in]
\node (a) at (-3,3){\begin{varwidth}{5cm}{
\begin{tabular}{|c|c|c|}
\cline{1-3}
$\st 1$ & $\st 2$ & $\st 3,-\!4$ \\
\cline{1-3}
$\st 4$ & $\st 5$ & \multicolumn{1}{c}{} \\
\cline{1-2}
\end{tabular}}
\end{varwidth}};

\node (b) at (1,3){\begin{varwidth}{5cm}{
\begin{tabular}{|c|c|c|}
\cline{1-3}
$\st 1$ & $\st 2$ & $\st 3,-\!4$ \\
\cline{1-3}
$\st 4$ & $\st 5$ & \multicolumn{1}{c}{} \\
\cline{1-2}
\end{tabular}}
\end{varwidth}};

\node (c) at (-6,1){\begin{varwidth}{5cm}{
\begin{tabular}{|c|c|}
\hline
$\st 1$ & $\st 2$ \\
\hline
$\st 3,-\!3,4$ & $\st 5$\\
\hline
\end{tabular}}
\end{varwidth}};

\node (d) at (-3,1){\begin{varwidth}{5cm}{
\begin{tabular}{|c|c|c|}
\cline{1-3}
$\st 1$ & $\st 2$ & $\st  3,-\!3$ \\
\cline{1-3}
$\st 4$ & $\st 5$ & \multicolumn{1}{c}{} \\
\cline{1-2}
\end{tabular}}
\end{varwidth}};

\node (e) at (1,1){\begin{varwidth}{5cm}{
\begin{tabular}{|c|c|}
\hline
$\st 1$ & $\st 2$ \\
\hline
$\st 4$ & $\st 5$\\
\hline
\end{tabular}}
\end{varwidth}};

\node (f) at (-6,-1){\begin{varwidth}{5cm}{
\begin{tabular}{|c|c|}
\hline
$\st 1$ & $\st 2$ \\
\hline
$\st 3,-\!4,4$ & $\st 5$\\
\hline
\end{tabular}}
\end{varwidth}};

\node (g) at (1,-1){\begin{varwidth}{5cm}{
\begin{tabular}{|c|c|}
\hline
$\st 1$ & $\st 2$ \\
\hline
$\st 3$ & $\st 5$\\
\hline
\end{tabular}}
\end{varwidth}};

\node (twiddle) at (-4.5, 1) {$\bs \sim$};

\path[thick, black] ([yshift=-3pt]a.south) edge ([yshift=3pt]d.north);
\path[thick, black] ([yshift=-3pt]b.south) edge ([yshift=3pt]e.north);
\path[thick, black] ([yshift=-3pt]c.south) edge ([yshift=3pt]f.north);
\path[thick, black] ([yshift=-3pt]e.south) edge ([yshift=3pt]g.north);
\path[thick, black] (-1,-2) edge (-1,4);

\end{tikzpicture}
\caption{On the left, four pretableaux in two adjacent pairs.  The two in the middle row are equivalent.  This induces the length 2 path on the three pontableaux shown on the right.  }
\label{fig:adjacency}
\end{figure}

\begin{prop}\label{p:structure_of_bn'}
Let $\sigma$ be any skew shape with $n$ boxes.  The augmented Brill-Noether graph $BN'(\sigma)$ is obtained from the Brill-Noether graph $BN(\sigma)$ in Definition \ref{d:bngraph} by:
\begin{enumerate}[(i)]
\item replacing each edge of $BN(\sigma)$,  between two aSYTs say $T$ and $T'$, by a path of $m$ edges, where $m$ is the difference between the two nonconcordant entries of $T$ and $T'$, and
\item attaching a path of length $\ell$ to each vertex $T$ for each way to add the number $i$ missing from $T$ to create a standard filling of $\sigma^a$ or ${}^a\sigma$ for some $a$.  The length $\ell$ of the path is determined as follows:
\begin{itemize}
\item if the new shape is of the form $\sigma^a$, then $\ell=n+1-i$.
\item if the new shape is of the form ${}^a\sigma$, then $\ell=i-1.$
\end{itemize}
\end{enumerate}
\end{prop}

\begin{proof}
We begin by describing a set of paths in $BN'(\sigma)$. We will show that the interiors of these paths are all disjoint, and that the union of the paths is the entire graph $BN'(\sigma)$. 
Let $T$ be an aSYT of shape $\sigma$, and let $b$ be a box that is either in $\sigma$ or in some ${}^a\sigma$ or $\sigma^a$.
 Let $i$ be a label from $\{1,2,\cdots,n+1\}$ such that box $b$ can be labeled with $i$ (erasing the existing label if $b$ is in $\sigma$) 
 and the result is an aSYT of shape $\sigma$ or a SYT of shape $\sigma^a$ or ${}^a\sigma$.
  Let $k$ be an integer, obeying the following constraint: if $b$ lies in $\sigma$, then $k$ is the existing label of $b$; if $b$ lies in ${}^a\sigma$ but not $\sigma$, 
  then $k$ must be $1$; if $b$ lies in $\sigma^a$ but not $\sigma$, then $k$ must be $n+1$. 
  Let $J$ consist of all integers between $i$ and $k$ inclusive, and for all $j \in J$ let $T_j$ denote the pretableau defined by adding the labels $-j$ and $k$ to box $b$.
   The pretableaux $T_j$ are all nonequivalent, hence they give distinct pontableaux. Also, when arranged in order by $j$, they form a path in $BN'(\sigma)$. 
   The length (number of edges) in this path is $|i-k|$. Each set of data $T,b,i,k$ defines such a path.

By Definition~\ref{d:pretableau}, every pretableau occurs in such a path. 
Furthermore, observe that the interiors of any two distinct paths are pairwise disjoint,
 because a pretableau in which the negative label does not match a positive label in the same box is not equivalent to any other pretableau. 
 Therefore the graph $BN'(\sigma)$ is formed by taking all the endpoints of these paths and adding paths of edges between them.
These paths fall into three classes according to whether $b$ lies in $\sigma$, left of $\sigma$, or right of $\sigma$.
 If $b$ lies in $\sigma$, then the two endpoints $T_i$ and $T_k$ are equivalent to two aSYT which are adjacent in $BN(\sigma)$,
  and which differ by replacing the label $i$ by the label $k$. The length of this path is $|i-k|$.
   In the second case, if $b$ lies to the left of $\sigma$, then the path has the aSYT $T$ on one end, the other end is not an aSYT, and the length of the path is $i-1$. 
   In the third case, if $b$ lies to the right of $\sigma$, then the path has $T$ on one end, the other end is not an aSYT, and the length of the path is $n+1-i$.
    Therefore the structure of the graph $BN'(\sigma)$ is as claimed.
\end{proof}

The data encoded in a pontableau with $r+1$ rows is equivalent to the data of a particular type of lattice path in $\mathbb{Z}^{r+1}$, which we will call a \emph{valid sequence.} We now describe these objects and the correspondence between them and pontableaux.

We will write $e_0,\ldots,e_{r}$ for the standard basis vectors of $\mathbb{Z}^{r+1}$, and for any $a\in\mathbb{Z}$, we will let $\ul a = (a,\ldots,a)\in \mathbb{Z}^{r+1}$.  
Write $|v|$ for the sum of the coordinates of a vector $v$, and write $v \ge w$ for vectors $v$ and $w$ if the inequality holds in each entry.

\begin{defn}\label{d:validseq}
Fix integers $g,r,d,$ a nondecreasing tuple $\alpha\in\mathbb{Z}_{\ge0}^{r+1}$, and a nonincreasing tuple $\beta\in\mathbb{Z}_{\ge0}^{r+1}$.  
A {\em valid sequence} for the data $g,r,d,\alpha,$ and $\beta$ is a sequence $$\boldsymbol \alpha = (\alpha^1,\ldots,\alpha^{g+1})$$ 
of nondecreasing tuples $\alpha^i \in\mathbb{Z}_{\ge0}^{r+1}$, satisfying
\begin{enumerate}[(i)]
\item $\alpha^1 = \alpha$, \label{it:alpha}
\item $\alpha^{g+1} = {\ul d}-{\ul r}-\beta$, and \label{it:beta}
\item for each $i=1,\ldots,g$, there exists some index $a=0,\ldots,r$ such that $$\alpha^{i+1}-\alpha^i \ge \ul 1 - e_a.$$  \label{it:bound}
\end{enumerate}
Denote by $\VS(g,r,d,\alpha,\beta)$ the set of all valid sequences for the data $g,r,d,\alpha,\beta$.
\end{defn}

We can regard $\boldsymbol \alpha$ as a lattice walk in the region of $\mathbb{Z}^{r+1}$ whose points have nondecreasing coordinates.  Then we will call the $g$ differences $\alpha^{i+1}-\alpha^{i}$ the {\em steps} of the walk $\boldsymbol \alpha$. We define the {\em progress} of a step to be 
$$|\alpha^{i+1}-\alpha^i|-r.$$

Note that this is a nonnegative number since $\alpha^{i+1}-\alpha^i$ is bounded below by some $\ul 1-e_a$ and $|\ul 1-e_a|=r.$ 

\begin{lemma}\label{l:total_progress}
Suppose $\bs \alpha \in \VS(g,r,d,\alpha,\beta)$.  Then the total progress of $\bs \alpha$ is $\rho = \rho(g,r,d,\alpha,\beta)$.  
\end{lemma}

\begin{proof}
By definition, for each $i=1,\ldots,g$, there exists an index $a_i$, depending on $i$, and a nonnegative vector $s^i$ such that $$\alpha^{i+1}-\alpha^i = \ul 1 - e_{a_i} + s^i.$$   
So,  the progress of the $i^{th}$ step is $|s^i|$.  Then 
$$\alpha^{g+1}-\alpha^1 = \sum_{i=1}^g (\ul 1 - e_{a_i} + s^i),$$
whereas by Definition~\ref{d:validseq}\eqref{it:alpha} and~\eqref{it:beta}, we have
$$\alpha^{g+1}-\alpha^1 = \ul d - \ul r - \beta - \alpha.$$
Therefore $$\sum_{i=1}^g (\ul 1 - e_{a_i} + s^i)=\ul d - \ul r - \beta - \alpha,$$
and we get
$$g(r+1)-g + \sum |s^i| = (r+1)(d-r) -|\alpha|-|\beta|.$$
Solving and using the definition of $\rho$, we conclude $\sum|s^i| = \rho$ as was claimed.
\end{proof}

Then the following corollary is immediate in the case $\rho=1$:
\begin{cor}\label{c:classification}
Suppose $\rho(g,r,d,\alpha,\beta) = 1$ and suppose $\bs \alpha \in \VS(g,r,d,\alpha,\beta)$.  Then 

\begin{enumerate}
\item There is exactly one index $j=1,\ldots,g$ for which $$\alpha^{j+1}-\alpha^{j} = \ul 1 + e_a-e_b,$$ for some indices $a,b \in\{0,\ldots,r\}$.  
If $a=b$, then we call the step a {\em stalling} step.  If $a\ne b$, then we call the step a {\em swapping} step.
\item For every other index $i \ne j$, we have 
$$\alpha^{i+1}-\alpha^{i} = \ul 1 - e_b$$ for some $b \in\{0,\dots,r\}$.  In this case we call the step a {\em plodding} step.
\end{enumerate}
\end{cor}

The following construction shows that we can conveniently enumerate valid sequences using pontableaux. 
Intuitively, a pontableau $T$ encodes a sequence of skew shapes. 
When $T$ is equivalent to an aSYT, this sequence consists of adding boxes one at a time, at the times encoded by their labels. 
In the case of pontableaux that are not aSYT, there is one step where one box is added while another is removed. 
To such a sequence of additions and removals of blocks, we associate a valid sequence, as follows. We will show that this construction gives a bijection.

\begin{defn}\label{d:pretableau-to-sequence}
Fix $g,r,d, \alpha, \beta$ and let $T$ be a pretableau on $\sigma(g, r, d, \alpha, \beta)$. 
Recall that we label the boxes of $\sigma$ with coordinates $(x,y)$ such that $(0,0)$ is the upper-left corner of the $(g-d+r) \times (r+1)$ rectangle
 (so that all boxes corresponding to $\alpha$ have negative $x$ coordinate). 
Let $b$ be any box in the same row as some box of $\sigma$, not necessarily contained in $\sigma.$
We will say that $T$ \emph{contains box $b$ after $i$ steps} if

\begin{itemize}
\item the box $b$ lies to the left of $\sigma$, and the total weight (see Definition~\ref{d:pretableau}(1)) of its labels of absolute value at most $i$ is $0$, or
\item the box $b$ lies in $\sigma$ or to the right of it, and the total weight of its labels of absolute value at most $i$ is $1$.
\end{itemize}

For each $i=1,\ldots,g+1$, let $\alpha^i$ be the following $(r+1)$-tuple of integers.
\begin{eqnarray*}
\alpha^i &=& (\alpha^i_0, \alpha^i_1, \cdots, \alpha^i_r)\\
\textrm{where } \alpha^i_j &=& i-1 - \min \{ x:\ \textrm{$T$ does not contain the box $(x,j)$ after $i-1$ steps} \}
\end{eqnarray*}
Call $\bs \alpha = (\alpha^1, \alpha^2, \cdots, \alpha^{g+1})$ the \emph{sequence associated to $T$}.
\end{defn}

\begin{ex}
We will describe how a valid sequence $\bs \alpha$ can be constructed from the following pretableau of shape $\sigma = (3,2)\setminus (1)$.

\medskip
{\centering
\begin{tabular}{|c|c|c|}
\cline{2-3}
\multicolumn{1}{c|}{} & $\st 1$ & $\st 3$\\
\cline{1-3}
$\st 2, -\!3, 4$ & $\st 5$ & \multicolumn{1}{|c}{}\\
\cline{1-2}
\end{tabular}
\par}
\medskip

Then the following sequence of six pictures (read the first row from left to right and then the second row from left to right) show which boxes are present after $i$ steps, 
for each $i$ from $0$ to $5$.
The boxes that are present are shaded gray.  (In principle, each shaded region would extend infinitely to the left.)

\medskip
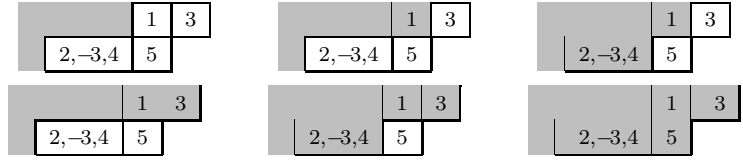
{\centering
\begin{tabular}{c|c|c|c|}
\cline{3-4}
\multicolumn{1}{c}{  \cellcolor{gray!50}  } &  \multicolumn{1}{c|}{\cellcolor{gray!50}} & $\st 1$ & $\st 3$\\
\cline{2-4}
\multicolumn{1}{c|}{\cellcolor{gray!50}} & $\st 2, -\!3, 4$ & $\st 5$ & \multicolumn{1}{|c}{}\\
\cline{2-3}
\end{tabular}
\qquad
\begin{tabular}{c|c|c|c|}
\cline{3-4}
\multicolumn{1}{c}{  \cellcolor{gray!50}  } &  \multicolumn{1}{c|}{\cellcolor{gray!50}} & \cellcolor{gray!50}$\st 1$ & $\st 3$\\
\cline{2-4}
\multicolumn{1}{c|}{\cellcolor{gray!50}} & $\st 2, -\!3, 4$ & $\st 5$ & \multicolumn{1}{|c}{}\\
\cline{2-3}
\end{tabular}
\qquad
\begin{tabular}{c|c|c|c|}
\cline{3-4}
\multicolumn{1}{c}{  \cellcolor{gray!50}  } &  \multicolumn{1}{c|}{\cellcolor{gray!50}} & $\cellcolor{gray!50}\st 1$ & $\st 3$\\
\cline{2-4}
\multicolumn{1}{c|}{\cellcolor{gray!50}} & $\cellcolor{gray!50}\st 2, -\!3, 4$ & $\st 5$ & \multicolumn{1}{|c}{}\\
\cline{2-3}
\end{tabular}
\par}

\medskip

{\centering
\begin{tabular}{c|c|c|c|}
\cline{3-4}
\multicolumn{1}{c}{  \cellcolor{gray!50}  } &  \multicolumn{1}{c|}{\cellcolor{gray!50}} & $\cellcolor{gray!50}\st 1$ & $\cellcolor{gray!50}\st 3$\\
\cline{2-4}
\multicolumn{1}{c|}{\cellcolor{gray!50}} & $\st 2, -\!3, 4$ & $\st 5$ & \multicolumn{1}{|c}{}\\
\cline{2-3}
\end{tabular}
\qquad
\begin{tabular}{c|c|c|c|}
\cline{3-4}
\multicolumn{1}{c}{  \cellcolor{gray!50}  } &  \multicolumn{1}{c|}{\cellcolor{gray!50}} & $\cellcolor{gray!50}\st 1$ & $\cellcolor{gray!50}\st 3$\\
\cline{2-4}
\multicolumn{1}{c|}{\cellcolor{gray!50}} & $\cellcolor{gray!50}\st 2, -\!3, 4$ & $\st 5$ & \multicolumn{1}{|c}{}\\
\cline{2-3}
\end{tabular}
\qquad
\begin{tabular}{c|c|c|c|}
\cline{3-4}
\multicolumn{1}{c}{  \cellcolor{gray!50}  } &  \multicolumn{1}{c|}{\cellcolor{gray!50}} & \cellcolor{gray!50}$\st 1$ &\cellcolor{gray!50} $\st 3$\\
\cline{2-4}
\multicolumn{1}{c|}{\cellcolor{gray!50}} &\cellcolor{gray!50} $\st 2, -\!3, 4$ & \cellcolor{gray!50}$\st 5$ & \multicolumn{1}{|c}{}\\
\cline{2-3}
\end{tabular}
\par}
\medskip

Note that the fourth of these pictures illustrates the origin of our use of the word ``swap:'' the previous shaded box $(-1,1)$ is swapped for the box $(1,0)$ during this step.

From these, we can read the following valid sequence by examining the $x$ coordinate of the unshaded box that is furthest to the left in each row. This gives the following six pairs of integers.
\medskip

{\centering
$\begin{array}{lll}
(0,-1) & (1,-1) & (1,0)\\
(2,-1) & (2,0) & (2,1)
\end{array}$
\par}

\medskip

Now subtracting these pairs (as vectors) from $(0,0), (1,1), \cdots, (5,5)$, we obtain the valid sequence $\bs \alpha$.

\medskip

{\centering
$\begin{array}{lll}
\alpha^1 = (0,1) & \alpha^2 = (0,2) & \alpha^3 = (1,2)\\
\alpha^4 = (1,4) & \alpha^5 = (2,4) & \alpha^6 = (3,4)
\end{array}$
\par}

\medskip
\end{ex}

The next definition allows us to describe when a pair of valid sequences correspond to adjacent pontableaux.

\begin{defn}\label{d:offby1}
Let us say that two valid sequences $\bs \alpha$ and $\bs \alpha'$ are {\em off by 1} if
$$\bs \alpha - \bs \alpha' = (\ul 0, \ldots, \ul 0, \pm e_a, \ul 0, \ldots, \ul 0)$$
for some index $a\in \{0,\ldots,r\}$.  In other words, all but one of the pairs of corresponding entries agree, and the two that disagree do so in a single coordinate with magnitude 1.
\end{defn}

\begin{lemma}\label{l:pont-valid} Let $\sigma$ be the skew shape associated to $g,r,d,\alpha,\beta$, where $\rho(g,r,d,\alpha,\beta) = 1$.
\begin{enumerate}
\item The sequence associated to any pretableau $T$ is a valid sequence.
\item Every valid sequence arises from some pretableau.
\item Two pretableaux have the same associated sequence if and only if they are equivalent.
\item Two pontableaux are adjacent if and only if their corresponding valid sequences are off by $1$.
\end{enumerate}
\end{lemma}

\begin{proof}
Let $T$ be a pretableau on $\sigma$ with associated sequence $\bs \alpha$, and let $i$ be any element of $\{1,2,\cdots,g\}$. 
Let $S_i$ and $S_{i+1}$ be the sets of boxes contained in $T$ after $i$ steps and $i+1$ steps, respectively. 
Then $S_{i+1}$ is obtained from $S_i$ by removing any boxes with the label $-(i+1)$ and then adding any boxes with the label $(i+1)$. 
It follows from this that $\alpha^{i+1} - \alpha^i$ can be computed by beginning with $\underline{1}$, adding a standard basis vector $e_a$ if the label $-(i+1)$ occurs in row $a$, 
and then subtracting the standard basis vector $e_b$ if the label $(i+1)$ occurs in row $b$. 
This shows that $\bs \alpha$ meets criterion (iii) in the definition of a valid sequence. Criteria (i) and (ii) of the definition are easy to verify. This establishes part 1 of the lemma.

Now, suppose that $\bs \alpha$ is a valid sequence. 
Then we can construct a pretableau $T$ from it by examining each vector 
$\underline{1} - \alpha^{i+1} + \alpha^i$ in turn and using it to place the label $(i+1)$ and possibly the label $-(i+1)$. 
Assume inductively that we have already placed the labels of absolute value at most $i$ in $T$.    

By Corollary \ref{c:classification}, the vector $\ul 1 - \alpha^{i+1}+\alpha^i$ is equal to either $e_b$ for some $b$, or $e_b-e_a$ for some $a$ and $b$. 
 In the first case, place label $(i+1)$ in the leftmost box of row $b$ that is not yet contained in $T$ (as in Definition~\ref{d:pretableau-to-sequence}). 
  In the second case, we do the following two operations, in either order:
\begin{itemize}
\item Place label $(i+1)$ in the leftmost box of row $b$ that is not yet contained in $T$. 
\item Place label $-(i+1)$ in the first box of row $a$ that is not yet contained in $T$.
\end{itemize}
  Again, the notion of containment is exactly as in Definition~\ref{d:pretableau-to-sequence}.  
  Because each vector $\alpha^i$ is nondecreasing, this placement of labels is guaranteed to produce a pretableau. 
  This establishes part $2$ of the lemma.  Note that the two operations above commute unless $a=b$.

For part $3$ of the lemma, observe that the only case in which the valid sequence does not \emph{uniquely} define the pretableau is when there is a step where $\alpha^{i+1} - \alpha^i = \underline{1}$, so that the indices $a$ and $b$ (as used above) must be equal to each other, but are not uniquely determined (and the order of placing $(i+1)$ and $-(i+1)$ matters as well). In this case, observe that choosing different values of $a=b$, and choosing the order to place $(i+1)$ and $-(i+1)$, corresponds to placing both of the labels $(i+1)$ and $-(i+1)$ together in various boxes of $T$, which give equivalent pretableaux by definition.

For part $4$ of the lemma, observe that two pretableaux are adjacent (hence their corresponding pontableaux are adjacent) 
if and only if one is obtained from the other by changing (in place) a label $-i$ to $-(i+1)$.
For any index $j \neq i$, the set of boxes contained after $j$ steps is the same for both pretableaux; these sets differ only after $i$ steps. 
The only difference is that, after $i$ steps,  one row of one of the pretableaux has one more box  than the other. 
This means precisely that the corresponding valid sequence of one is obtained by subtracting a unit basis vector from one element of the valid sequence of the other; 
hence the two valid sequences are adjacent. The same reasoning in reverse shows that two valid sequences that differ by $1$ arise from adjacent pretableaux, as desired.
\end{proof}

Combining all of the statements of Lemma~\ref{l:pont-valid}, we have:

\begin{cor}\label{cor:last1}
When $\rho(g,r,d,\alpha,\beta) = 1$, the set of pontableaux is in bijection with $\VS(g,r,d,\alpha,\beta)$, via the construction in Definition \ref{d:pretableau-to-sequence}, and two pontableaux are adjacent in $BN'(\sigma)$ if and only if their valid sequences are off by 1.  
\end{cor}

 Finally, it will be useful to record the following observation, which is a corollary of the proof of Lemma~\ref{l:pont-valid}:

\begin{cor}\label{cor:last2}
Pontableaux that are almost-standard Young tableaux missing a given index $i$ 
correspond to valid sequences in which the $i^\text{th}$ step is a stall and the other steps are plods.  
Pontableaux with labels $i$ and $-i$ in different boxes correspond to valid sequences in which the $i^\text{th}$ step is a swap and the other steps are plods. 
\end{cor}

\section{The scheme structure of $\Grdab(E,p,q)$ on an elliptic curve} \label{sec:elliptic}

In this section, we recall some  definitions and results on the scheme of linear series $\Grd(C)$ on a smooth curve $C$, and on $\Grdab(C,p,q)$,
the twice-pointed version of this scheme. 
We will then fix a single elliptic curve $E$, with two marked points $p,q$ such that the divisor $p-q$ is not a torsion point of $\Pic^0(E)$, and study $\Grdab(E,p,q)$
 when $\rho(1,r,d,\alpha,\beta) \in \{0,1\}$. We will prove that $\Grdab(E,p,q)$ is isomorphic as a scheme to either a (reduced) point, a $\PP^1$, or $E$ itself.  

  
Let $C$ be a smooth, proper, connected curve of genus $g$ and let $r$ and $d$ be positive integers with $g-d+r \geq 0$.
The Brill-Noether locus $W^r_d(C)$ is a closed subset of the Picard variety $\mbox{Pic}^d(C)$ consisting of line bundles $L$ with at least $r+1$ independent global sections. 
Its natural desingularization $G^r_d(C)$ parametrizes \emph{linear series} on $C$,
 that is, pairs $(L,V)$, such that $L$ is a degree $d$ line bundle on $C$ and $V\subseteq H^0(C,L)$ is an $(r+1)$-dimensional vector subspace of global sections of $L$.
A linear series of degree $d$ and (projective) dimension $r$ is usually called a $g^r_d$. We recall that the space of linear series on a curve can be given a natural scheme structure as a determinantal variety (see \cite[IV.3]{ACGH}).

One can extend this definition to allow for fixed ramification profiles. 
Given a point $p\in C$ and a linear series $(L,V)$, the \emph{vanishing sequence} of $\cL =(L,V)$ at $p$ is the strictly increasing sequence
$$a_0(\cL,p) <\cdots<a_r(\cL,p)$$ 
of the $r+1$ distinct orders of vanishing at $p$ achieved by the sections of $V$. 
The \emph{ramification sequence} $\alpha=(\alpha_0,\dots,\alpha_r)$ is the non-decreasing sequence defined as $\alpha_i=a_i-i$.
 We will also use the letters $b$ and $\beta$ to denote vanishing orders (resp. ramification orders) listed in \textit{decreasing} (resp. non-increasing) order, at a second point $q$. 
 That is, the numbers
$$b_0(\cL,q) > b_1(\cL,q) > \cdots > b_r(\cL,q)$$
will denote the vanishing orders at $q$, and $\beta = (\beta_0,\beta_1,\cdots,\beta_r)$ will be the ramification orders in non-increasing order: $\beta_i = b_i + i - r$.
 Our reason for listing these numbers in different orders at the two different points is that it will substantially declutter several definitions,
  such as the compatibility condition for limit linear series in the following section.

\begin{defn}\label{def:grdab} Let $C$ be a nonsingular projective curve of genus $g$, and let $p,q$ be distinct points of $C$.
Let $r,d$ be nonnegative integers. Let $\alpha$ be a non-decreasing and $\beta$ a non-increasing sequence of $r+1$ nonnegative integers.
We write $\Grdab(C,p,q)$ for the space of $g^r_d$s with ramification sequence at least $\alpha$ (respectively $\beta$) at the point $p$ (respectively $q$), 
with their natural  scheme structure, as explained in \cite[IV.3]{ACGH}.\end{defn}
We will refer to $\Grdab(C,p,q)$ as the \emph{Brill-Noether scheme.}
 From the construction of  $\Grdab(C,p,q)$ as a determinantal variety, together with the dimensions of Schubert cycles corresponding to $\alpha,\beta$ in a suitable Grassmannian, its expected dimension is  the {\em adjusted Brill-Noether number}
$$\rho= \rho(g, r, d, \alpha, \beta)= g-(r+1)(g-d+r)-\sum _i\alpha_i-\sum _i\beta_i.$$
Moreover, every component has at least this expected dimension.

 Another way to define the Brill Noether scheme is via the functor of families of $g^r_d$s with specified ramification that it represents.  
In the case of an elliptic curve $E$, by which we always mean a smooth proper curve of genus 1, and $p,q$ geometric points of $E$ such that $p-q$ is not torsion in $\Pic^0(E)$, we now describe the functor of points of $\Grdab(E,p,q)$.  Observe that the scheme $\Pic^d(E)$ has a degree $d$ vector bundle $\cH$, whose fiber over a closed point $[L] \in \Pic^d(E)$ is naturally identified with the vector space of sections of $L$. The scheme $G^r_d(E)$ (with no imposed ramification) is isomorphic to the Grassmannian bundle of $(r+1)$-planes in $\cH$. To impose ramification, consider two flags in $\cH$:
   
   \begin{eqnarray*}
   	\cH = \cP_0 \supset \cP_1 \supset \cdots \supset \cP_{d-1}\\
	\cH = \cQ_0 \supset \cQ_1 \supset \cdots \supset \cQ_{d-1},
   \end{eqnarray*}   
   where the fibers of $\cP_i$ over $[L]$ correspond to global sections of $L$ vanishing to order at least $i$ at $p$, and the fibers of $\cQ_i$ correspond to sections vanishing to order at least $i$ at $q$.
   
   \begin{defn} \label{def:family}
   	Let $S$ be any scheme. Then a \emph{family over $S$ of $g^r_d$s on $E$ with ramification at least $\alpha$ at $p$ and $\beta$ at $q$} consists of data $(\ell, \cV)$, where $\ell:\ S \rightarrow \Pic^d(E)$ is a morphism, and $\cV$ is a locally free rank $r+1$ subsheaf of $\ell^\ast \cH$ such that for all $i \in \{0,1,\cdots,r\}$, the induced bundle maps
	\begin{eqnarray*}
		\cV &\rightarrow& \ell^\ast (\cH / \cP_{i + \alpha_i})\\
		\cV &\rightarrow& \ell^\ast (\cH / \cQ_{i + \beta_{r-i}})
	\end{eqnarray*}
	each have rank less than or equal to $i$.
   \end{defn}
   
   The scheme structure of $\Grdab(E,p,q)$ is characterized by the following universal property: its functor of points is the functor associating to any scheme $S$ the set of isomorphism classes of families of $g^r_d$s on $E$ with ramification at least $\alpha$ at $p$ and $\beta$ at $q$. See \cite[Theorem 4.1.3]{osserman} for further details, and a proof that this functor is indeed representable.
   
   Our basic tool in analyzing $\Grdab(E,p,q)$ as a scheme is the following lemma, which makes it possible to reduce the analysis to a few simple cases. In the statement below, we use the notation $a_i = \alpha_i + i$ and $b_i = \beta_i + r-i$ (these are the vanishing orders corresponding to the given ramification order) in order to make the statements more concise.
   
   \begin{lemma}\label{splitLemma}
Suppose that $k$ is an index such that one of the following two conditions holds:
\begin{itemize}
\item $a_{k+1} + b_{k} \geq d+1$, or
\item $a_{k+1} + b_{k} = d$, and the fiber of the map $\Grdab(E,p,q) \rightarrow \Pic^d(E)$ over the point corresponding to the line bundle $\cO_E\left( a_{k+1} p + b_{k} q \right)$ is empty.
\end{itemize}
Then the scheme $\Grdab$ is isomorphic to the fiber product
$$G^{r',\alpha',\beta'}_d \times_{ \Pic^d } G^{r'',\alpha'',\beta''}_d$$
where the superscripts in this expression are defined as follows.

\begin{center}$
\begin{array}{rcl c rcl}
r' &=& k && r'' &=& r-k-1\\
\alpha'_i &=& \alpha_i && \alpha''_i &=& \alpha_{k+1+i} + (k+1)\\
\beta'_i &=& \beta_{i} + (r-k) && \beta''_i &=& \beta_{k+1+i}
\end{array}
$\end{center}

In addition, for each closed point $(L,V)$ in $\Grdab$, the corresponding closed points $(L,V'),(L,V'')$ have disjoint sets of vanishing orders at $p$ (resp., $q$), whose union is the set of vanishing orders of $(L,V)$ at $p$ (resp, $q$).
\end{lemma}

The ramification sequences $\alpha', \beta'$ are determined by the vanishing sequences shown in the following table. Informally, the sequences $\alpha',\beta'$ are chosen so that both vanishing sequences can be split into two non-interacting halves.

\begin{center}
$\begin{array}{|c|c|c|}\hline
 & G^{r',\alpha',\beta'}_d & G^{r'',\alpha'',\beta''}_d\\\hline
 \textrm{vanishing orders at $p$}&
 \{a_0, a_1,  \dots, a_k\}&
 \{a_{k+1}, a_{k+2}, \dots, a_r \}
 \\\hline
 \textrm{vanishing orders at $q$}&
 \{b_0, b_{1}, \dots, b_{k} \}&
 \{b_{k+1}, b_{k+2}, \dots, b_r\}
 \\ \textrm{(reverse order)} &&\\\hline
\end{array}$
\end{center}
   
   \begin{proof}[Proof of Lemma \ref{splitLemma}]
   We first informally summarize the argument. Given a closed point $(L,V)$ of $\Grdab(E,p,q)$, the conditions of the lemma guarantee that the two subspaces $V(-a_{k+1}p)$ and $V(-b_{k} q)$
 of the vector space $V$ intersect trivially. 
 Since their dimensions add up to at least $r+1$, it follows that $V$ is the direct sum of these two subspaces. These two subspaces constitute the corresponding points of $G^{r',\alpha',\beta'}_{d'}(E,p,q)$ and $G^{r'',\alpha'',\beta''}_{d'}(E,p,q)$. Notice that these two linear series uniquely determine the original, by taking the sum of the two vector spaces.
 We prove the lemma by showing that this decomposition can be carried out in families.
 
Let $\cV$ be the tautological subbundle on $\Grdab$ (this corresponds, under the above functorial description, to the bundle on $\Grdab$ given by the identity map on $\Grdab$), and $\ell: \Grdab \rightarrow \Pic^d(E)$ the forgetful map. 
Let $k$ be the index mentioned in the statement of the lemma. Consider the induced morphisms of vector bundles
\begin{eqnarray*}
f_1:\ \cV &\rightarrow& \ell^\ast (\cH / \cP_{a_{k+1}})\\
f_2:\ \cV &\rightarrow& \ell^\ast (\cH / \cQ_{b_{k}})
\end{eqnarray*}

We claim that the kernel of $f_1$ is a vector bundle of rank $r-k$, and the kernel of $f_2$ is a vector bundle of rank $k+1$. 
To see this, notice that over any closed point $(L,V)$ of $\Grdab$, the kernel of $f_1$ can be identified with $V(-a_{k+1} p )$, which is a vector space of dimension at least $r-k$, and the kernel of $f_2$ can be identified with $V(-b_k q)$, of dimension at least $k+1$. These two spaces intersect trivially, so their dimensions must be exactly $r-k$ and $k+1$, respectively. Therefore $f_1$ has rank $k+1$ at all closed points, and in a local trivialization all of its $(k+2) \times (k+2)$ minors vanish; this shows that the kernel of $f_1$ is indeed a vector bundle of rank $r-k$. Similarly, the kernel of $f_2$ is a vector bundle of rank $k+1$.

The map $\ell$ together with the vector bundle $\ker f_1$ constitutes a family of $g^{r'}_{d'}$s with ramification at least $\alpha',\beta'$ at $p$ and $q$, hence we obtain a map $\Grdab(E,p,q) \rightarrow G^{r',\alpha',\beta'}_{d'}(E,p,q)$. We obtain a similar map from $\ker f_2$, hence taking these together gives a map
$$F:\ \Grdab \rightarrow G^{r',\alpha',\beta'}_d \times_{ \Pic^d } G^{r'',\alpha'',\beta''}_d.$$
It remains to show that $F$ is an isomorphism. To do this, it suffices to check that it satisfies the universal property of fiber products. 
Suppose that $S$ is any scheme, with two maps $\ell':\ S \rightarrow G^{r',\alpha',\beta'}_d$ and $\ell'':\ S \rightarrow G^{r'',\alpha'',\beta''}_d$ 
such that the two composition maps to $\Pic^d(E)$ are the same; call this map $\ell:\ S \rightarrow \Pic^d(E)$. 
The two maps $\ell', \ell''$ define two subbundles $\cV', \cV''$ of $\ell^\ast \cU$. 
These two subbundles intersect trivially, because otherwise there would be a closed point $[L]$ of $\Pic^d(E)$ 
and two linear series $(L,V'),(L,V'')$ such that $V' \cap V''$ contains a nonzero section of $L$ vanishing along the divisor $a_{k+1}p + b_{r-k} q$, contradicting the hypothesis of the Lemma.
Therefore their sum $\cV = \cV' + \cV''$ is a rank $r+1$ subbundle of $\ell^\ast \cH$. Now, the bundle map
$$\cV \rightarrow \ell^\ast (\cH / \cP_{a_i})$$
has rank at most the sum of the ranks of the maps from $\cV'$ and $\cV''$ to $\ell^\ast(\cH /  \cP_{a_i})$, which is at most $r+1 - i$. 
Similar remarks hold for the map to $\ell^\ast(\cH/  \cQ_{b_i})$. Therefore $\cV$ defines a morphism $S \rightarrow \Grdab$. 
This is the unique map that factors the two maps from $S$ to $G^{r',\alpha',\beta'}_d$ and $G^{r'',\alpha'',\beta''}_d$. 
Therefore $F$ satisfies the necessary universal property, and must be an isomorphism.
   \end{proof}

The next well-known result gives the conditions on $\alpha$ and $\beta$ for which $\Grdab(E,p,q)$ is nonempty. A general form of this lemma is also included in \cite[Lemma 2.1]{osserman-simpleproof}. We continue to suppose that $E$ is an elliptic curve; $p,q\in E$ are points such that $[p-q]$ is not torsion in $\Pic^0(E)$; $r$ and $d$ are nonnegative integers; and $\alpha = (\alpha_0,\ldots,\alpha_r)\in\mathbb{Z}^{r+1}_{\ge 0}$, respectively $\beta=(\beta_0,\ldots,\beta_r)\in\mathbb{Z}^{r+1}_{\ge 0}$, is a nondecreasing, respectively nonincreasing, sequence.

\begin{lemma}\label{lem:is_nonempty}

The scheme $\Grdab(E,p,q)$ is nonempty if and only if there exists some $t\in\{0,\ldots,r\}$ such that
$$\alpha + \beta \le \ul d - \ul r - \ul 1 + e_t.$$
\end{lemma}

\begin{proof}
Suppose $(L,V)$ is a point in $\Grdab(E,p,q)$.  
For each $t\in \{ 0,\dots, r\}$, the space of sections in $V$ with ramification at least $\alpha_t$ at $p$ has dimension at least $r-t+1$, 
and
the space of sections with ramification at least $\beta_t$ at $q$ has dimension at least $t+1$.
Since $\dim V = r+1$, there exists some section in $V$ satisfying the two ramification conditions. 
The orders of vanishing at these points of such a section are at least $\alpha_t+t$ and  $\beta_t+r-t$. 
As the degree of $L$ is $d$, this implies $\alpha_t+\beta_t+r\le d$.
Furthermore, the genericity of the points $p$ and $q$ implies that equality is obtained at most for one value of $t$.
The conditions on $\alpha$ and $\beta$ can thus  be written
$$ \alpha +\beta \le \ul d - \ul r - \ul 1 +e_t.$$

Conversely, suppose $t\in\{0,\ldots,r\}$ is such that $\alpha +\beta \le \ul d - \ul r - \ul 1 +e_t.$ 
 Increasing the numbers in $\alpha$ and $\beta$, we may as well assume that $\alpha +\beta = \ul d - \ul r - \ul 1 +e_t$, 
 since if the new conditions give a nonempty scheme then the original conditions did too. 
  Recall that we write $a_i = \alpha_i + i$ and $b_i = \beta_i + r-i$ for the vanishing orders corresponding to $\alpha$ and $\beta$.
    So $a_t + b_t = d$ and $a_j + b_j = d-1$ for all $j\in\{0,\ldots,r\}$ with $j\ne t$. 

Let $L \cong \cO(a_t p + b_t q)$; then for each $j=0,\ldots, r$ including $j=t$, there exists a unique, up to scaling, nonzero section $s_j\in H^0(L)$
 with vanishing order exactly $a_j$ at $p$ and exactly $b_j$ at $q$, since there exists a unique, up to scaling, nonzero section of 
 $L(-a_jp - b_j q)\cong \cO((a_t\!-\!a_j)p + (b_t\!-\!b_j)q))$ and there are no sections of $L(-(a_j+1)p - b_j q), L(-a_jp - (b_j +1)q)$
 for each $j=0,\ldots,t$.  Furthermore all the sections $s_j$ are clearly independent,
  since they have distinct orders of vanishing at $p$ and $q$, so we have exhibited a point $(L,V=\langle s_0,\ldots,s_r\rangle)$ in $\Grdab(E,p,q)$.  
\end{proof}
   We will require the following simple case for our next argument.
   
   \begin{lemma} \label{l:r0}
   	Suppose that $a,b,d$ are such that $d - a- b \geq 0$. Then
	$$G^{0,(a),(b)}_{d} (E,p,q) \cong G^0_{d-a-b}(E) \cong \textrm{Sym}^{d-a-b}(E).$$
	In particular, 
	\begin{itemize}
	\item if $d - a - b \geq 1$ then $G^{0,(a),(b)}_d(E,p,q)$ is a $\PP^{d-a-b-1}$-bundle over $\Pic^d(E)$, and 
	\item if $d -a - b = 0$ then $G^{0,(a),(b)}_d(E,p,q)$ is a single reduced point, mapping to $[ \cO_E( a p + b q) ] \in \Pic^d(E)$.
	\end{itemize}
   \end{lemma}
   
   \begin{proof}
   	The first isomorphism follows from the definition of a family of $g^r_d$s with ramification. The second isomorphism is standard; see for example \cite[VII, Proposition 2.1]{ACGH}.
   \end{proof}

   We now prove our desired results on $\Grdab(E,p,q)$ when $\rho=0$ or $1$ in the following proposition.
   
   \begin{prop} \label{oneCurveV2}
   Suppose $d,r,\alpha,\beta$ are chosen so that $\Grdab(E,p,q)$ is nonempty. 
   \begin{enumerate}
   \item 
   	If $\rho(g,r,d,\alpha,\beta)=1$, then:
	\begin{enumerate}
	\item Either there exists $k\in\{0,\ldots,r\}$ such that $\alpha_k +\beta_k = \ul d - \ul r$, in which case $\Grdab(E,p,q) \cong \PP^1$, or
	\item for every $j\in\{0,\ldots,r\}$, $\alpha_j + \beta_j = \ul d - \ul r - 1$, in which case $\Grdab(E,p,q) \cong E$.
	\end{enumerate}
	\item If $\rho(g,r,d,\alpha,\beta)=0$, then $\Grdab(E,p,q)$ is a reduced point.
	\end{enumerate}
   \end{prop}

   \begin{proof}
   	As before, we will use the notation $a_i = \alpha_i +i,\ b_i = \beta_i + r-i$ to denote the vanishing orders corresponding to the given ramification orders. 
Note that $$\rho(1,r,d,\alpha,\beta) = 1+\sum_{i=0}^r (d-r-1-\alpha_i-\beta_i).$$
By assumption, $\Grdab(E,p,q)$ is nonempty, so from Lemma~\ref{lem:is_nonempty} it follows that each of the $r+1$ terms 
\begin{equation}\label{eq:term}
d-r-1-\alpha_i-\beta_i,\quad i = 0,\ldots,r
\end{equation} is nonnegative, except possibly for one term which may be $-1$.  

Now suppose $\rho = 1$.  Our analysis implies that we have the following two cases: 

   	\textit{Case (a):} One of the $r+1$ terms is $-1$, one of them is $1$, and the rest are $0$. In other words, there exist distinct indices $j,k\in\{0,\ldots,r\}$ such that 
	\begin{equation}\label{eq:3cases}
	\alpha_i + \beta_i = \begin{cases}
	d-r& \text{if }i=k,\\
	d-r-2& \text{if }i=j,\\
	d-r-1&\text{otherwise.}\end{cases}
	\end{equation}
We claim that the hypotheses of Lemma \ref{splitLemma} hold for every index $i \in \{0,1,\cdots,r-1\}$. First observe that for all such $i$, $a_{i+1} + b_i$ is greater than both $a_i + b_i$ and $a_{i+1} + b_{i+1}$. By~\eqref{eq:3cases}, at least one of these two quantities is greater than or equal to $d-1$. So $a_{i+1} + b_i \geq d$ for all $i\in\{0,\ldots,r-1\}$. Furthermore, when equality holds we necessarily have $i \neq k$, so $(a_{i+1} p + b_i q) \neq (a_k p + b_k q)$, so the fiber of $\Grdab(E,p,q)$ over $[\cO_E(a_{i+1}p + b_i q)]$ is empty. Therefore we apply Lemma \ref{splitLemma} repeatedly to obtain
	$$\Grdab \cong G^{0, (a_0),(b_0)}_d \times_{\Pic^d} \cdots \times_{\Pic^d} G^{0,(a_r),(b_r)}_d.$$
	It now follows from Lemma \ref{l:r0} that the $j^\text{th}$ factor in this fiber product is a $\PP^1$-bundle over $\Pic^d(E)$, the $k^\text{th}$ factor is a reduced point mapping to $[\cO_E(a_k p + b_k q)]\in\Pic^d(E)$, and all other factors map isomorphically to $\Pic^d(E)$.  It follows that  $\Grdab$ is isomorphic to $\PP^1$ as a scheme.
	
	\textit{Case (b):} Each of the $r+1$ terms in~\eqref{eq:term} are $0$.  In other words, $\alpha_i + \beta_i = d-r-1$ for all $i$. In this case we claim that $\Grdab(E,p,q) \cong \Pic^d(E)$ scheme-theoretically. We argue by induction on $r$; the base case $r=0$ is trivial. For larger $r$, we consider two cases. First, suppose that $a_{i+1} > a_i + 1$ for some $i \in \{0,1,\cdots,r-1\}$. Then in this case we can apply Lemma \ref{splitLemma} to express $\Grdab(E,p,q)$ as a fiber product of $\Pic^d(E)$ of two schemes, which (by the inductive hypothesis) both map isomorphically to $\Pic^d(E)$, and the result follows. 
	
	In the other case, $a_{i+1} = a_i + 1$ for all $i$. In this case, it follows that $\Grdab(E,p,q) \cong G^r_{d-a_0 - b_r}(E)$, and a short calculation shows that $d-a_0 - b_r = r+1$. The result now follows from the fact that $G^r_{r+1}(E) \cong \Pic^{r+1}(E)$.

Now for part (2) of the Proposition, suppose $\rho =0$.  Then it must be the case that one of the terms in \eqref{eq:term} is $-1$ and the rest are $0$.  In other words, there exists $k\in\{0,\ldots,r\}$ such that 
	\begin{equation}\label{eq:2cases}
	\alpha_i + \beta_i = \begin{cases}
	d-r& \text{if }i=k,\\
	d-r-1&\text{otherwise.}\end{cases}
	\end{equation}
Then the same reasoning from Case (a) above applies: we deduce that the hypotheses of Lemma~\ref{splitLemma} hold for each $i\in\{0,\ldots,r-1\}$, so we again apply that Lemma repeatedly to obtain
	$$\Grdab \cong G^{0, (a_0),(b_0)}_d \times_{\Pic^d} \cdots \times_{\Pic^d} G^{0,(a_r),(b_r)}_d.$$
From here, using Lemma~\ref{l:r0} it follows again that the $k^\text{th}$ factor of this fiber product is a reduced point mapping to $[\cO_E(a_k p + b_k q)]\in\Pic^d(E)$, while every other factor maps to $\Pic^d$ isomorphically. We conclude that $\Grdab(E,p,q)$ is a reduced point mapping to $[\cO_E(a_k p + b_k q)]\in\Pic^d(E)$.
   \end{proof}

\section{Limit linear series on an elliptic chain}
\label{sec:prelim}

Having studied linear series on twice-pointed elliptic curves in the previous section, in this section we study limit linear series on chains of elliptic curves of genus $g$, and relate the scheme of such limit linear series to our earlier combinatorial constructions. After defining limit linear series on chains of elliptic curves, we will give a construction that associates to each valid sequence $\bs \alpha$ (and therefore, when $\rho = 1$, to each pontableau) a set $C(\bs \alpha)$ of limit linear series, and prove (Corollary \ref{c:union_of_c}) that 
$$
\Grdab(X,p,q) = \bigcup_{\bs \alpha \in \VS(g,r,d,\alpha,\beta)} C(\bs \alpha).
$$
We then specialize to the case $\rho = 1$, and obtain specific information about the geometry of each $C(\bs \alpha)$ and show that their incidence relations are encoded by the augmented Brill-Noether graph of Section~\ref{sec:ponttab}.


 First we introduce limit linear series and the Eisenbud-Harris scheme structure on them.  
A proper, connected nodal curve is said to be of {\em compact type} if its Jacobian is compact or equivalently, if its dual graph has no loops. 
The right notion of linear series for compact type curves is the concept of  limit linear series (see \cite{eh-lls}):
 \begin{defn}\label{def:limitgrd}
Let $C$ be a curve of compact type with irreducible components $\{C_i\}$. 
A \emph{limit linear series} of degree $d$ and rank $r$ on $C$, also called a {\em limit $\grd$} on $C$,
 is the data of a linear series $\cL_i=(L_i,V_i)$ of degree $d$ and rank $r$ on each $C_i$, 
subject to the following condition. Given a node of $C$ obtained by gluing points $p_i\in C_i$ and $p_j\in C_j$, 
denote by $a_t(\cL_i, p_i)$ the orders of vanishing at $p_i$ of the linear series $(L_i,V_i)$ written in increasing order and by 
$b_{t}(\cL_j,p_j)$ the orders of vanishing at $p_j$ of  $(L_j,V_j)$ written in decreasing order, for each $t=0,\ldots,r$. 
Then
  $$a_t(\cL_i, p_i)+b_{t}(\cL_j,p_j) \ge d \quad \mbox{ for each }\quad t=0,\dots, r.$$
A limit $g^r_d$ is called \emph{refined} if at each node, each of the $r+1$ 
inequalities above are equalities. Otherwise it is called {\em coarse}.  The pair $(L_i, V_i)$ is called the
\emph{aspect} of the series on the component $C_i$. 
\end{defn}   

We again write $\Grd(C)$ for the limit $\grd$s on the compact type curve $C$ and 
$\Grdab(C,p,q)$ for the limit $\grd$s on a twice-marked $(C,p,q)$, where $p,q$ are smooth points of $C$, with ramification profiles $\alpha, \beta$.

Now we define generic chains of elliptic curves, which will be the degeneration that we use in the rest of our arguments.

 \begin{defn}\label{def:chain} Let $E_1,\dots,E_g$ be elliptic curves, with $p_i$ and $q_i$ distinct points on $E_i$ for each $i$.
 Glue $q_i$ to $p_{i+1}$ for $i = 1,\dots, g-1$, to form a nodal curve that we call a \emph{chain of elliptic curves} $X$ of genus $g$.
  If $p_i-q_i$ is not a torsion element in $\Pic^0(E_i)$ for any $i$, we say that $X$ is a {\em generic elliptic chain}.
  \end{defn}
 Throughout, we shall also consider $X$ as a point of $\overline{\mathcal{M}}_{g,2},$ with two marked points $p=p_1$ and $q=q_g$; 
 we will refer to such an $(X,p,q)$ as a {\em twice-marked generic elliptic chain. }
We specify once and for all a scheme structure on $\Grdab(X,p,q)$ that we will henceforth refer to as the Eisenbud-Harris scheme structure. 
 (It can be defined for the limit linear series on any compact type curve, as in \cite{eh-lls}; we restrict our definition to this case purely to ease the notation.)

\begin{defn}\label{def:eh-structure}
Let $(X,p,q)$ be a twice-marked elliptic chain as in~\ref{def:chain}. Choose $g, r, d, \alpha, \beta$ as in the setup of Definition~\ref{def:grdab}.
 The {\em Eisenbud-Harris scheme} $\Grdab(X,p,q)$ is the scheme obtained as the union
\begin{equation}\label{eq:union}
\bigcup \prod_{i=1}^g G^{r,\alpha^i,\beta^i}_d (E_i,p_i,q_i) \subseteq \prod_{i=1}^g \Grd(E_i)
\end{equation}
where the union is over all choices of ramification profiles $(\alpha^1,\beta^1,\ldots,\alpha^g,\beta^g)$
with the property that 
\begin{itemize}
\item $\alpha^1 = \alpha$ and $\beta^g = \beta$, and
\item 
for each $i=1,\ldots,g-1$ and each $t=0,\ldots,r$, we have
$$\beta^i_t + \alpha^{i+1}_{t} = d-r.$$
\end{itemize}
\end{defn}
\noindent In other words, we take the union over all possible {\em refined} ramification profiles on $X$ that have ramification exactly $\alpha$ at $p=p_1$ and $\beta$ at $q=q_g$.
 All limit linear series  are present in the locus defined by at least one of these ramification profiles. 
 For coarse series, one may construct profiles $\alpha^i, \beta^i$ containing that series by decreasing the required ramification of the linear series at the nodes arbitrarily so that  the inequalities become an equalities.

Our next goal is to describe $\Grdab(X,p,q)$ when $\rho=1$ (see Definition~\ref{def:eh-structure}).
  We shall show that its components have the structure of elliptic and rational curves, 
  and that they intersect according to the augmented Brill-Noether graph defined in Definition~\ref{d:ponadjacency}.

 Recall from Definition~\ref{d:validseq} 
 the definition of a valid sequence $\bs \alpha\in \mathrm{VS}(g,r,d,\alpha,\beta)$ associated to $g,r,d,\alpha,$ and $\beta$.
 Furthermore, recall Corollary~\ref{c:classification} and the terminology therein: when $\rho=1$, in a valid sequence 
$\bs \alpha$ exactly one of the $g$ steps of $\bs \alpha$ is either a {\em stalling} step or a {\em swapping} step, and all of the other steps are {\em plodding} steps.

We show how for any $\rho$, the components of the Eisenbud-Harris scheme $\Grd(X,p,q)$ are in bijection with valid sequences.  
     
\begin{defn}\label{def:C(alpha)}
Fix $g\ge 1,$ let $r$ and $d$ be nonnegative integers, $\alpha$ a non-decreasing and $\beta$ a non-increasing $(r+1)$-sequence of non-negative integers.
  Let $\bs \alpha = (\alpha^1,\ldots,\alpha^{g+1})$ consist of $g+1$ sequences of nondecreasing $(r\!+\!1)$-tuples with 
$\ul 0 \le \alpha^i \le \ul d - \ul r,$ and $\alpha^1=\alpha$.
Define the {\em complementary sequence} $\bs \beta = (\beta^0,\ldots,\beta^g)$ by $$\beta^i = \ul d - \ul r - \alpha^{i+1},$$ and assume 
that $\beta^g = \beta$.  Then we define the scheme
$$C(\bs \alpha) = \prod_{i=1}^g G_d^{r,\alpha^i,\beta^i}(E_i,p_i,q_i) \subseteq \prod \Grd(E_i).$$
\end{defn}

Note that the condition for $\bs \alpha$ to be a valid sequence, namely $\alpha^{i+1}-\alpha^i \ge \ul 1 - e_{a(i)} $ for each $i$, is equivalent to
 $\alpha^{i}+\beta^i \le \ul d-\ul r- \ul 1 + e_{a(i)} $ for each $i$. So, from Lemma \ref{lem:is_nonempty}, we have

\begin{cor}\label{l:nevalid} With the notations in definition \ref{def:C(alpha)}, the scheme $C(\bs \alpha)$ is nonempty
 if and only if $\bs \alpha$ is a valid sequence for the data $g,r,d,\alpha,\beta$.
\end{cor}

We thus obtain the following description of $\Grdab(X,p,q)$, valid for all values of $\rho(g,r,d,\alpha,\beta)$.

\begin{cor} \label{c:union_of_c}
If $(X,p,q)$ is a generic twice-pointed chain of elliptic curve of genus $g$, the Eisenbud-Harris scheme $\Grdab(X,p,q)$ as in Definition~\ref{def:eh-structure} is  
$$\bigcup_{\bs \alpha \in \VS(g,r,d,\alpha,\beta)} \!\! C(\bs \alpha),$$
Furthermore, the refined limit linear series are precisely those points which lie in only \textit{one} of the loci $C({\bs \alpha})$.
\end{cor}

\begin{proof}
Given a limit linear series $\cL$ with aspect $\cL_i$ on $E_i$, observe that $\cL \in C({\bs \alpha})$ for a valid sequence $\bs \alpha=(\alpha^1,\ldots,\alpha^{g+1})$ if and only if the following inequalities hold for all $i, j= 1,\ldots,g$:
$$
\underline{d}-\underline{r} - \beta^{i}(\cL_{i},q_{i}) \leq \alpha^{i+1} \quad\text{ and }\quad \alpha^{j} \leq  \alpha^{j}(\cL_{j},p_{j}).
$$ 
Setting $j=i+1$ and combining the inequalities shows that such a valid sequence $\bs \alpha$ can always be found, and that the choice of $\bs \alpha$ is unique if and only if $\cL$ is refined.
\end{proof}

We now focus on the case relevant to our present purpose: when $\rho(g,r,d,\alpha,\beta) = 1$. In this case, we can also give a specific classification of the geometry of all of the loci $C({\bs \alpha})$.

\begin{cor}\label{l:factors_of_c}
Let $\bs \alpha = (\alpha^1,\ldots,\alpha^{g+1})$ be a valid sequence for the data $g,r,d,\alpha,\beta$.
\begin{enumerate}
\item If $\rho(g,r,d,\alpha,\beta)=1$, then  $C(\bs \alpha) $  is isomorphic to the elliptic curve $E_i$ via the  projection to $\mbox{Pic}^d(E_i)\cong E_i$
 if the $i^{th}$ step of $\bs \alpha$ is a stall, and isomorphic to $\mathbb{P}^1$ if the $i^{th}$ step of $\bs \alpha$ is a swap. 
 (The projection to $\mbox{Pic}^d(E_i)$ in this case is a point.)
\item  If $\rho(g,r,d,\alpha,\beta)=0$, then  $C(\bs \alpha) $  is a single reduced point. 
\end{enumerate}
\end{cor}

\begin{proof}Let $(\beta^0,\ldots,\beta^g)$ be the complementary sequence to $\bs \alpha$.
The Brill-Noether number for refined limit linear series is additive over the components of the curve.   If $\rho(g,r,d,\alpha,\beta) = 1$, then the Brill-Noether numbers $$\rho_1 = \rho(1,r,d,\alpha^1,\beta^1),\,\,\ldots,\,\,\rho_g = \rho(1,r,d,\alpha^g,\beta^g)$$ for the components are all zero except for one of them, say $\rho_i$, which is one.  Then the result follows from Proposition~\ref{oneCurveV2}.
 
 If $\rho(g,r,d,\alpha,\beta) = 0$, each of the Brill-Noether numbers $\rho_1,\ldots,\rho_g$ for the individual components $E_i$ are all zero and part (2) of Proposition~\ref{oneCurveV2} applies.
\end{proof}

\begin{lemma}\label{l:refined_and_coarse} Fix $g,r,d,\alpha,$ and $\beta$ such that $\rho(g,r,d,\alpha,\beta) = 1$. Let $(X,p,q)$ be a generic twice-pointed elliptic chain of genus $g$.
\begin{enumerate}
\label{it:refined_unique}
\item If $\mathcal{L} \in \Grdab(X,p,q)$ is a coarse limit linear series. 
Then there exist exactly two valid sequences $\bs \alpha$ and $\bs \alpha'$ such that $\mathcal{L} \in C(\bs \alpha)\cap C(\bs \alpha')$. 
 Furthermore $\bs \alpha$ and $\bs \alpha'$ are off by 1.
\label{it:coarse}
\item 
Suppose that $\bs \alpha$ and $\bs \alpha'$ are valid sequences that are off by 1. 
 Then there exists a unique coarse limit linear series $\mathcal{L} \in C(\bs \alpha)\cap C(\bs \alpha')$.
 \item If $C(\bs \alpha)$ is rational, then it intersects one other component if it projects nontrivially (to more than a point) to $\Grd(E_i)$ for $i= 1$ or $i=g$, and two other components if it projects nontrivially to $\Grd(E_i)$ for $i\in\{2,\ldots,g-1\}$.
\label{it:coarse_converse}
\end{enumerate}
\end{lemma}

\begin{proof} We will use the classification in Proposition~\ref{oneCurveV2} repeatedly on the aspects of various limit linear series.  
 Suppose $\bs \alpha$ is a valid sequence.  From Corollary \ref{c:classification} and Proposition \ref{oneCurveV2},  the aspects of a linear series in  $C(\bs \alpha)$ are completely determined  
 on all but one of the elliptic components.
On that remaining component, say   $E_i$, one has  $\alpha^{i+1}-\alpha^{i} = \ul 1 + e_{a(i)}-e_{b(i)}$ or equivalently  
$\alpha^{i} +\beta^{i} = \ul d- \ul r- \ul 1 +e_{b(i)}-e_{a(i)}$, for indices $a(i),b(i)\in\{0,\ldots,r\}$.
If a series in  $ C(\bs \alpha)$ is not refined, then the strict inequality can only occur at either  $p_i$ or at $q_i$,
 where the ramification of  $\mathcal{L}$ must be increased either from the  given $\alpha^i$ at $p_i$ or $\beta^i = \ul d-\ul r-\alpha^{i+1}$ by $e_{a(i)}$. Furthermore,
 by Proposition~\ref{oneCurveV2}(2), there is only one such coarse series $\cL$ with these ramifications.
  If $a(i)\not= b(i)$, then $a(i)$ is uniquely determined, so the only choice is either $p_i$ or $q_i$ for $i\not=1,g$ and $q_1, p_g$ respectively for $i=1,g$. 
 The condition $a(i)\not= b(i)$ corresponds to  $C(\bs \alpha)$ being a $\mathbb{P}^1$.
 Therefore, a rational component of the Brill-Noether curve contains one coarse limit linear series if the rational component lies above $E_1$ or $E_g$ and two otherwise.  This will show (3), once we show (1).  This also shows that for a coarse series, the strict inequality for the gluing orders occurs at a single node and it is off by one.
 
Assume now  $\mathcal{L} $ is a coarse series and the strict inequality occurs at the node obtained from gluing $q_{i-1}$ and $p_i$ for the  index $a$.
 For $j=2,\ldots,g$, denote by $\hat\alpha^{j}$ the ramification orders of $\mathcal{L} $ at  $p_j,$ and for $j=1,\ldots,g-1$, denote by $\hat\beta^{j}$ the ramification orders of $\cL$ at  $q_j$.
 Our assumption is 
 $$\hat\beta^{j-1}=\begin{cases}\ul d-\ul r-\alpha^{j}&  j\not=i, \\ \hat\beta^{j-1}=\ul d-\ul r-\alpha^{j}+e_a, & j=i.\end{cases}$$
Then $\mathcal{L} $ belongs to exactly two components  $C(\bs \alpha), C(\bs \alpha ')$ where $\bs \alpha$ is given by $\alpha^j=\hat\alpha^j$ for all $j$,  while $\bs \alpha'$ is given by 
$\alpha^{'j}=\hat\alpha^j$ if $j\not= i$ and $\alpha^{'i}=\hat\alpha^i-e_a$. By construction, $\bs \alpha, \bs \alpha '$ are off by one. This proves (1).

Conversely, assume that $\bs \alpha$ and $\bs \alpha'$ are valid sequences that are off by 1. Let $k\in\{1,\ldots,g\}$ be the index such that $\alpha^i=\alpha ^{'k}-e_a$, and $\alpha^j=\alpha ^{'j}$ for all $j\not=k$.  Write $\bs \beta'$ for the complementary sequence to $\bs \alpha'$.  Then a limit linear series lies in $C(\bs \alpha)\cap C(\bs\alpha')$ if it has ramification at least $\bs \alpha$ at the points $p_i$ and $\bs \beta'$ at the points $q_i$.
As $\rho=1$ and both $\bs \alpha, \bs \alpha '$ are valid sequences, Corollary  \ref{c:classification}
implies that $\bs \alpha$ has a stall or swap in the $(k-1)^\text{st}$ step, $\bs \alpha'$ has a stall or swap in the $k^\text{th}$ step, and for all $j=1,\ldots,g$, we have
$$\alpha^{j}+{\beta'}^{j} = \begin{cases}
\ul d-\ul r- \ul 1 + e_{a} &\text{if }j=  k,\\
\ul d-\ul r- \ul 1 + e_{b(j)} &\text{if }j\ne k,\end{cases}$$ 
for indices $b(j)$ depending on $j$. 
Then, from Proposition \ref{oneCurveV2}, there exists a unique limit linear series with ramifications $ \bs \alpha '$ at the points $p_i$ and  $\bs \beta$ at the points $q_i$, and by Corollary~\ref{c:union_of_c} it is necessarily coarse. This proves (2).
\end{proof}

\begin{thm}\label{t:structure_of_bn_curve}
Fix $g,r,d,\alpha,\beta$ with $\rho(g,r,d,\alpha,\beta) = 1$.  Let $\sigma = \sigma(g,r,d,\alpha,\beta)$ be the skew shape as in Definition \ref{def:skewtab}.

The scheme $\Grdab(X,p,q)$ is a reduced nodal curve 
whose dual graph is $BN'(\sigma)$.  Moreover,
\begin{itemize}
\item A vertex of $BN'(\sigma)$ that is an aSYT missing the number $i$, say, corresponds to a component $C(\bs \alpha)$ which is isomorphic as a scheme to the elliptic curve $E_i$.
\item The remaining vertices of $BN'(\sigma)$ correspond to components $C(\bs \alpha)$ that are isomorphic as a scheme to $\mathbb{P}^1$.
\end{itemize}
\end{thm}

\begin{proof}
This follows by combining Corollaries~\ref{cor:last1},~\ref{cor:last2}, \ref{c:union_of_c}, \ref{l:factors_of_c}(1), and Lemma~\ref{l:refined_and_coarse}(1) and (2). The only thing we need to verify is that the components intersect nodally.  But this follows by observing that whenever two components $C(\bs \alpha)$ and $C(\bs \alpha')$ meet, they each vary in only one of the $g$ factors of $\prod G^r_d(E_i)$, and these two factors are different.  Indeed, it was shown in the proof of Lemma~\ref{l:refined_and_coarse}(2) that if $\bs \alpha$ and $\bs \alpha'$ differ in index $k$, then $C(\bs \alpha)$ varies only in the $(k-1)^\text{st}$ factor while $C(\bs \alpha')$ varies only in the $k^\text{th}$ factor.
\end{proof}

\begin{cor}\label{WrdYT}   The projection $W^{r,\alpha,\beta}_d (X,p,q)$ to the Jacobian is a nodal curve with all components elliptic, whose dual graph is the Brill-Noether graph $BN(\sigma(g,r,d,\alpha,\beta))$ in Definition~\ref{d:bngraph}.  A vertex in $BN(\sigma)$ corresponding to a tableau missing the entry $i$ corresponds to a component isomorphic to $E_i$.
\end{cor}
\begin{proof}The image $W^{r,\alpha,\beta}_d (X,p,q)$ of  $G^{r,\alpha,\beta}_d (X,p,q)$  in the Jacobian 
is obtained by contracting the rational components, by Corollary \ref{l:factors_of_c}. Now the result follows from Theorem \ref{t:structure_of_bn_curve}.
\end{proof}

We conclude this section by observing that the proof of Theorem \ref{t:structure_of_bn_curve} applies, with very slight modifications, to give the following description in the case $\rho(g,r,d,\alpha,\beta) = 0$. Note that the number of limit linear series on a generic curve with given ramification at two given points was previously computed by Tarasca (see \cite[Section 3.1]{t}).

\begin{prop}\label{p:rho0}
Let $g,d,r,\alpha,\beta$, and $\Grdab(X,p_1,q_g)$ be as above, but now assume that $\rho(g,d,r,\alpha,\beta) = 0$. Then $\Grdab(X,p_1,q_g)$ consists of a collection of reduced points, each given by a refined limit linear series. These points correspond bijectively to the standard fillings of the skew shape $\sigma(g,d,r,\alpha,\beta)$.
\end{prop}

\begin{proof}
Corollaries \ref{c:union_of_c} and \ref{l:factors_of_c} show that, as a scheme, $\Grdab(X,p,q)$ is a finite set of points in bijection with the valid sequences ${\bs \alpha}$ corresponding to $\alpha,\beta$. When $\rho = 0$, the argument of Lemma \ref{l:pont-valid} can be adapted to show that these valid sequences are in bijection with the standard Young tableaux on $\sigma(g,r,d,\alpha,\beta)$; the only difference is that there are no swaps or stalls at all.
\end{proof}

\section{Proofs of the main theorems}\label{sec:final}
We now combine our results to describe the scheme of Eisenbud-Harris limit linear series with specified ramification, in the case where the adjusted Brill-Noether number 
$$g - (r+1)(g-d+r) - |\alpha| - |\beta|$$ is equal to $1$. 

\begin{prop}\label{t:mainchain}
Fix $g,r,$ and $d$, and let $\alpha = (\alpha_0,\ldots, \alpha_r)$ be nondecreasing  and $\beta = (\beta_0,\ldots,\beta_r)$ be nonincreasing sequence of non-negative integers. 
 Let $\sig = \sig(g,r,d,\al,\be)$ be the skew shape defined by $(g,r,d,\al,\be)$ as in Definition \ref{def:skewtab}.
Suppose that the adjusted Brill Noether number is
$$\rho(g,r,d,\al,\be) = g-(r+1)(g-d+r)-|\al|-|\be| = 1.$$
Then for a twice-marked elliptic chain of genus $g$ that is generic in the sense of Definition \ref{def:chain}, the scheme $\Grdab(X,p,q)$ is a reduced nodal curve of arithmetic genus 
$$1+(r\!+\!1)(n\!+\!1)f^\sigma + \sum_{i=1}^{r+1} (r\!+\!1\!-\!i)\!\cdot\! f^{{}^i\sigma} - \sum_{i=1}^{r+1} (r\!+2\!-\!i)\!\cdot\! f^{\sigma^i}.$$
\end{prop}

\begin{proof}
The scheme $G^{r, \alpha^i, \beta^i}_d(E_i,p_i,q_i)$ was shown to be a nodal curve whose dual graph is the augmented Brill-Noether graph $BN'(g,r,d,\alpha,\beta)$; that it is reduced is shown in Proposition~\ref{oneCurveV2}.  

The genus of $\Grdab(X,p,q)$ is unchanged by the operation of contracting chains of $\PP^1$s.  
The nodal curve thus obtained has elliptic components only.  Then, by Corollary \ref{WrdYT}
 its dual graph is precisely the Brill Noether graph $\BN(\sigma(g, r, d, \alpha, \beta))$ (see Definitions \ref{d:bngraph}, \ref{def:skewtab}). 

Combining these facts, the arithmetic genus of $\Grdab(X,p,q)$ is
\begin{eqnarray*}
&& |V(\BN(\sig))| + (|E(\BN(\sig))| - |V(\BN(\sig))| + 1)\\[.1cm]
&=& 1+|E(\BN(\sig))|\\[-.1cm]
&=& 1+(r\!+\!1)(n\!+\!1)f^\sigma + \sum_{i=1}^{r+1} (r\!+\!1\!-\!i)\!\cdot\! f^{{}^i\sigma} - \sum_{i=1}^{r+1} (r\!+2\!-\!i)\!\cdot\! f^{\sigma^i}.
\end{eqnarray*}
where the last equality follows directly from the computation of the number of edges of $\BN(\sig)$ in Theorem~\ref{t:masterformula}.
\end{proof}

\begin{thm}\label{t:main}
Fix $g,r,d,\alpha,\beta$ as in Theorem~\ref{t:mainchain}, so that $\rho(g,r,d,\al,\be)=1$.  Then for a general twice-pointed smooth curve $(X,p,q)$ of genus $g$, the scheme $\Grdab(X,p,q)$ is an at-worst-nodal curve of arithmetic genus 
\begin{equation}
1+(r\!+\!1)(n\!+\!1)f^\sigma + \sum_{i=1}^{r+1} (r\!+\!1\!-\!i)\!\cdot\! f^{{}^i\sigma} - \sum_{i=1}^{r+1} (r\!+2\!-\!i)\!\cdot\! f^{\sigma^i}.
\end{equation}
\end{thm}

The following enumerative result (previously found by Tarasca \cite{t} without explicit reference to skew tableaux) will follow from the same proof.

\begin{thm}\label{t:main3}
Fix $g,r,d,\alpha,\beta$, and assume that
$$\rho(g,r,d,\alpha,\beta) = 0.$$
Then for a general twice-pointed smooth curve $(X,p,q)$ of genus $g$, the scheme $G^{r,\alpha,\beta}_d(X,p,q)$ consists of $f^\sigma$ points.
\end{thm}

\begin{proof}[Proof of Theorems \ref{t:main} and \ref{t:main3}]
Consider a family of curves parameterized by the spectrum of a discrete valuation ring whose special fiber $X_0$ is a generic chain
of elliptic curves and whose generic fiber $X_{\eta}$ is a non-singular curve. The dimension of $\Grdab(X,p,q)$ is the expected number 
$\rho$, the space of refined limit linear series is dense, and the Eisenbud-Harris scheme structure is reduced. 
Then, from Corollary 3.4 in \cite{murrayosserman}, the family of Eisenbud-Harris limit linear series is flat and proper over $B$. 
 Note that while  \cite{murrayosserman} does not explicitly consider specified ramification, the methods in that paper can be extended to cover also the case of ramification points
 (compare with \cite{osserman-llsmodsch} Def 4.5). 
The special fiber of the family is the Eisenbud-Harris space of limit linear series on $X_0$,
while the generic fiber is the classical space of linear series on the non-singular curve $X_{\eta}$. 
In our situation, the fibers of the family are curves, and the arithmetic genus of the curves in a flat proper family is constant.
Hence the result follows from Theorem \ref{t:mainchain}.
\end{proof}

Now the Eisenbud-Harris and Pirola genus formula is a special case of our results:
\begin{cor}\label{cor:final}
Suppose $\rho(g,r,d) = 1$.  For a general smooth curve $X$ of genus $g$, the genus of the curve $\Grd(X)$ is 
$$1+\frac{(r+1)(g-d+r)}{g-d+2r+1}\cdot g!\cdot \prod_{i=0}^r \frac{i!}{(g-d+r+i)!}.$$
\end{cor}

\begin{proof}
This follows from combining Theorem~\ref{t:harmonic} and Theorem~\ref{t:main}.
\end{proof}

\end{document}